\documentclass[11pt]{article}

\usepackage[margin=1.25in]{geometry}

\usepackage{amssymb,amsmath}
\usepackage{amsthm}
\usepackage{hyperref}
\usepackage{mathrsfs}
\usepackage{textcomp}
\hypersetup{
    colorlinks,%
    citecolor=black,%
    filecolor=black,%
    linkcolor=black,%
    urlcolor=black
}

\usepackage{verbatim}

\usepackage{sectsty}

\makeatletter

\newdimen\bibspace
\renewenvironment{thebibliography}[1]{%
 \section*{\refname 
       \@mkboth{\MakeUppercase\refname}{\MakeUppercase\refname}}%
     \list{\@biblabel{\@arabic\c@enumiv}}%
          {\settowidth\labelwidth{\@biblabel{#1}}%
           \leftmargin\labelwidth
           \advance\leftmargin\labelsep
           \itemsep\bibspace
           \parsep\z@skip     %
           \@openbib@code
           \usecounter{enumiv}%
           \let\p@enumiv\@empty
           \renewcommand\theenumiv{\@arabic\c@enumiv}}%
     \sloppy\clubpenalty4000\widowpenalty4000%
     \sfcode`\.\@m}
    {\def\@noitemerr
      {\@latex@warning{Empty `thebibliography' environment}}%
     \endlist}

\makeatother

\makeatletter

\newtheorem{thm}{Theorem}[section]
\newtheorem{lem}[thm]{Lemma}
\newtheorem{prop}[thm]{Proposition}

\newtheorem{cor}[thm]{Corollary}
\newtheorem{rem}[thm]{Remark}

\def\XXint#1#2#3{{\setbox0=\hbox{$#1{#2#3}{\int}$}
  \vcenter{\hbox{$#2#3$}}\kern-.5\wd0}}

\newcommand{\al}{\alpha}                \newcommand{\lda}{\lambda}
\newcommand{\om}{\Omega}                \newcommand{\pa}{\partial}
\newcommand{\va}{\varepsilon}           \newcommand{\ud}{\mathrm{d}}
\newcommand{\be}{\begin{equation}}      \newcommand{\ee}{\end{equation}}

\newcommand{\R}{\mathbb{R}}              
\newcommand{\Lt}{\mathcal{L}^{2}_t}

\begin{document}

\title{\textbf{Extinction profiles for  the  Sobolev critical fast diffusion equation in bounded domains. I. One bubble dynamics}
\bigskip}

\author{\medskip  Tianling Jin\footnote{T. Jin is partially supported by NSFC grant 12122120, and Hong Kong RGC grants GRF 16303822 and GRF 16306320.}, \  \
Jingang Xiong\footnote{J. Xiong is partially supported by NSFC grants 12325104 and 12271028.}}

\date{\today}

\maketitle

\begin{abstract}
In this paper, we investigate the extinction behavior of nonnegative solutions to the Sobolev critical fast diffusion equation in bounded smooth domains with the Dirichlet zero boundary condition. Under the two-bubble energy threshold assumption on the initial data, we prove the dichotomy  that every solution converges uniformly, in terms of  relative error,  to either a steady state or a blowing-up bubble.  

\end{abstract}

\section{Introduction}

Let $\om$ be a bounded domain in $\R^n$, $n\ge 1$, with smooth boundary $\pa \om$, and let $m\in (0,1)$. Consider the fast diffusion equation
\be \label{eq:rho-1}
\pa_t \rho =\Delta \rho^{m} \quad \mbox{in } \om \times (0,\infty)
\ee
with the Cauchy-Dirichlet boundary condition
\be\label{eq:rho-2}
\rho\big|_{t=0}= \rho_0 \ge 0,   \quad  \rho=0 \quad  \mbox{on }\pa \om  \times (0,\infty),
\ee
where $\Delta$ is the Laplace operator in  the spatial variables $x=(x^1,\dots, x^n)\in \R^n$, and $\rho_0 \in C_0^1(\om)$ does not vanish identically. The equation is singular near the zero set $\rho$.   Since the work of Sabinina \cite{S1}, it has been known that the Cauchy-Dirichlet problem has a unique bounded nonnegative weak solution, and the solution will be extinct after a finite time $T^*>0$. Namely, $\rho>0$ in $\om \times (0,T^*)$ but $\rho\equiv 0$ in $\om\times [T^*, \infty)$. Therefore, the equation is parabolic, and thus, smooth, in $\om \times (0,T^*)$. Since $\rho=0$ on $\partial\Omega\times (0,T^*)$, the equation is singular there. Chen-DiBenedetto \cite{CDi} proved that the solution is H\"older continuous on $\overline\om \times (0,T^*)$.  DiBenedetto-Kwong-Vespri \cite{DKV} obtained a global Harnack inequality for its solutions, and showed the spatial Lipschitz continuity of $\rho^m$ in $\overline\Omega$. In \cite{JX19, JX22}, we established the optimal regularity
\be \label{eq:reg-jx}
\pa_t^l  \rho^{m} \in 
\begin{cases}
C^{2+\frac{1}{m}} (\overline \om \times (0,T^*))\mbox{ if } \frac{1}{m} \mbox{ is not an integer}\\
C^{\infty} (\overline \om \times (0,T^*))\mbox{ if } \frac{1}{m} \mbox{ is  an integer}\\
\end{cases} 
\quad \forall~ l\ge 0,
\ee
which in particular solved the first  problem listed in Berryman-Holland \cite{BH}. Throughout this paper,  the solutions of \eqref{eq:rho-1} and \eqref{eq:rho-2} are the classical ones before the extinction time and satisfy the above regularity \eqref{eq:reg-jx}.

The extinction behavior near $T^*$ has been thoroughly characterized in the Sobolev subcritical regime $\frac{(n-2)_+}{(n+2)}<m<1$. Under the assumption that $\rho^{m} \in C^2(\overline \om\times (0,T^*))$,  Berryman-Holland \cite{BH} proved that the function $\rho^{m}$ converges to a separable solution along a sequence of times in $H^1_0(\Omega)$.  Feireisl-Simondon \cite{FS}  proved the uniform convergence without the regularity assumption. Later, Bonforte-Grillo-V\'azquez \cite{BGV} proved  the uniform convergence of the  relative error, and Bonforte-Figalli \cite{BFig} quantified the convergence rate of the relative error and obtained a sharp exponential rate in generic domains. Akagi \cite{Ak} provided a different proof of this sharp exponential convergence result. Such convergence of the relative error in the $C^2$ topology then follows from the regularity \eqref{eq:reg-jx}. We also showed  the polynomial convergence rate for all smooth domains  in \cite{JX20-a}.  More  recently, Choi-McCann-Seis \cite{CMS} proved that the relative error either decays exponentially with the sharp rate,  or else decays algebraically at a rate $1/t$ or slower. Furthermore, they obtained higher order asymptotics. See also the recent papers Choi-Seis \cite{CS} and Bonforte-Figalli \cite{BFig24} for more references.  Asymptotics of  \eqref{eq:rho-1} and \eqref{eq:rho-2} in the context of porous medium (where $m>1$) have  also been well studied; see Aronson-Peletier \cite{AP1981}, Jin-Ros-Oton-Xiong \cite{JRX} and references cited therein.

However, the critical and supercritical regimes $0<m\le \frac{(n-2)_+}{(n+2)}$ remain largely unexplored. The difficulty can be inferred from the associated elliptic problem. In \cite{P}, Pohozaev proved that the equation
\be \label{eq:elliptic}
-\Delta S= S^{p} \quad \mbox{in }\om,\quad
  S>0 \quad \mbox{in }\om,\quad
 S=0 \quad \mbox{on }\pa \om
\ee 
has no solution if $\om$ is star-shaped, where $p=\frac{1}{m}\ge \frac{n+2}{n-2}$ and $n\ge 3$. This is entirely distinct from the situation for $m$ in the Sobolev subcritical regime.

In this paper, we  focus on the  Sobolev critical regime: $m=\frac{1}{p}= \frac{n-2}{n+2}$ and $n\ge 3$. In contrast to the nonexistence result of Pohozaev in star-shaped domains,  existence of solutions to \eqref{eq:elliptic} was obtained by Kazdan-Warner \cite{KW} if $\om$ is an annulus, and by  Bahri-Coron \cite{Bahri-C} if $\om$ has a nontrivial  homology with $\mathbb{Z}_2$-coefficients.  In another direction, Br\'ezis-Nirenberg \cite{BN} proved existence of \eqref{eq:elliptic} in dimension  $n\ge 4$ when $\Delta $ replaced by $\Delta +b$, where $b$ is a positive constant smaller than the first Dirichlet eigenvalue of $-\Delta$ in $\om$.  Thus, we may analyze the extinction behavior when 

\begin{itemize}

\item[(i).] $\om$ is star-shaped, or more generally, there is no solution of \eqref{eq:elliptic} in $\om$ with $p=\frac{n+2}{n-2}$; 
 
\item[(ii).] Br\'ezis-Nirenberg's perturbation is imposed; 
 
 \item[(iii).] Bahri-Coron's topological condition is imposed.

\end{itemize}
 
In the first case, Galaktionov-King \cite{GKing} derived a sharp extinction profile when $\om$ is a ball and $\rho$ is radially symmetric without bubble towers.  However, the full characterization of the dynamics of the solutions remains unresolved for general domains, and is the main topic of the current paper. Nonetheless, Sire-Wei-Zheng \cite{SWZ} constructed a solution with an explicit extinction rate, and the normalized energy concentrates at a finite number of points. They used the parabolic gluing methods which aligns with those used by  Cort\'azar-del Pino-Musso \cite{CdM}  and D\'avila-del Pino-Wei \cite{DdW}. 
  Recently, the second case was fully addressed by the authors in \cite{JX20-a}.  
We established compactness and proved that the extinction phenomenon is parallel to that in the subcritical regime. The third case is more challenging. 

Let $K(n)$ be the best constant of the Sobolev inequality in $\R^n$. That is, 
\[
K(n):=\inf\left\{\frac{\int_{\R^n} |\nabla \eta|^2 \ud x}{ (\int_{\R^n} |\eta|^{\frac{2n}{n-2}} \,\ud x)^{\frac{n-2}{n}}}, \quad \forall~ \eta\in C_c^\infty(\R^n), ~\eta\not\equiv 0\right\}.
\]
It is well-known that $K(n)$ is achieved by 
\be \label{eq:U}
U_{a,\lda}=[ n(n-2)]^{\frac{n-2}{4}}   \left(\frac{\lda}{1+\lda^2|x-a|^2}\right)^{\frac{n-2}{2}}, \quad a\in \R^n, ~\lda>0,
\ee
which satisfies the equation
\be \label{eq:Uequation}
-\Delta U_{a,\lda}= U_{a,\lda} ^{\frac{n+2}{n-2}} \quad \mbox{in }\R^n. 
\ee
Furthermore, $\forall~ f\in H_0^1(\om), ~f\not\equiv 0$,
\begin{equation}\label{eq:energyfunctional}
Y_{\om}(f):= \frac{\int_{\om} |\nabla f|^2 \ud x}{ (\int_{\om} |f|^{\frac{2n}{n-2}} \,\ud x)^{\frac{n-2}{n}}} \ge K(n), 
\end{equation}
and $\inf \{Y_\om (u): f\in H_0^1(\om), ~f\not\equiv 0 \}= K(n)$ but is never achieved. 

Let $P U_{a,\lda}$ be the projection of $U_{a,\lda}$ into $H_0^1(\om)$ defined by
\begin{equation}\label{eq:projection}
PU_{a,\lda}= U_{a,\lda}- h_{a,\lda}, 
\end{equation}
where  
\be  \label{eq:h-a}
\begin{cases}
\Delta h_{a,\lda}(x) = 0 \quad \mbox{in }\om, \\  
h_{a,\lda}= U_{a,\lda} \quad \mbox{on }\pa \om. 
\end{cases}
\ee  
We will refer to both $U_{a,\lda}$ and $PU_{a,\lda}$ as bubbles centered at $a$.

Our main theorem is as follows. 

\begin{thm}\label{thm:main}
Let $\rho$  be a nonnegative solution of \eqref{eq:rho-1} and \eqref{eq:rho-2} with $m=\frac{n-2}{n+2}$ and $n\ge 3$, and let $T^*>0$ be its extinction time.  Let $\delta\in (0,T^*)$.  Assume $Y_\om (\rho_0^m) \le 2^{\frac{2}{n}}K(n)$. Then the following dichotomy holds:
\begin{itemize}
\item[(i).] There exist a solution $S$ of \eqref{eq:elliptic}  with $p=\frac{n+2}{n-2}$, and positive constants $C_1$ and $\gamma_1$ such that 
\be \label{eq:sr-infinity}
 \left\|\frac{(T^*-t)^{-\frac{n-2}{4}}\rho^m(\cdot, t)}{S}-\left(\frac{4}{n+2}\right)^{\frac{n-2}{4}}\right\|_{C^2(\overline \om)}\le C_1 |\ln (T^*-t)|^{-\gamma_1} \ \ \forall\, t\in(\delta, T^*).
 \ee 

\item[(ii).] There exist positive constants $C_2$ and $\gamma_2$ such that  
\be \label{eq:sr-1}
\left\|\frac{(T^*-t)^{-\frac{n-2}{4}}\rho^m(\cdot,t)}{PU_{a(t),\lda(t)}} - \left(\frac{4}{n+2}\right)^{\frac{n-2}{4}}\right\|_{L^\infty(\Omega)}\le C_2 |\ln (T^*-t)|^{-\gamma_2}\ \ \forall\, t\in(\delta, T^*), 
\ee 
where $a: (0,T^*)\to \om$ and $\lda:(0,T^*) \to [1,\infty) $ are smooth functions satisfying 
\begin{equation}\label{eq:bubblelimit}
\begin{split}
&\lim_{t\to (T^*)^-} |\ln (T^*-t)|^{\frac{1}{n-2}}  |a(t)-a_{T^*}|=0\mbox{ for some }a_{T^*}\in\Omega, \mbox{ and}\\  
&\lim_{t\to (T^*)^-} |\ln (T^*-t)|^{-\frac{1}{n-2}}  \lda(t) \mbox{ exists and is positive}.
\end{split}
\end{equation}
\end{itemize} 
The constants $C_1$ and $C_2$ depend only on $n,\Omega,\delta$ and $\rho_0$. The constants $\gamma_1$ and $\gamma_2$ depend only on $n,\Omega$ and $\rho_0$. 
\end{thm}

If there is no solution of \eqref{eq:elliptic} with $p=\frac{1}{m}=\frac{n+2}{n-2}$, such as when $\Omega$ is star-shaped, then we know that the normalized solution $(T^*-t)^{-\frac{n-2}{4}}\rho^m(\cdot,t)$ must blow up. Indeed, in Proposition \ref{prop:integral-crit} we showed the sharp blow-up criterion in  Lebesgue norms: $(T^*-t)^{-\frac{n-2}{4}}\|\rho^m(\cdot,t)\|_{L^{2n/(n-2)}(\Omega)}$ converges to a positive real number, but $(T^*-t)^{-\frac{n-2}{4}}\|\rho^m(\cdot,t)\|_{L^q(\Omega)} \to \infty$ as $t\to (T^*)^-$ for every $q>\frac{2n}{n-2}$. By using the concentration compactness arguments in Struwe \cite{St},   Br\'ezis-Coron \cite{BC} and Bahri-Coron \cite{Bahri-C}, we showed earlier in \cite{JX20-a} that the blow up must be in the form of the bubbles in \eqref{eq:projection}, and consequently, there must be exactly one bubble present under the assumption $Y_\om (\rho_0^m) \le 2^{2/n}K(n)$. This assumption is not perturbative in the sense that it does not necessarily imply that the initial data is close to a steady state or a bubble.

To  obtain the one bubble dynamics of the solution delineated in part (ii) of Theorem \ref{thm:main}, we make a change of variables, which transforms the fast diffusion equation \eqref{eq:rho-1} into the normalized Yamabe flow \eqref{eq:main} on the smooth bounded domain $\Omega$ with the Dirichlet zero boundary condition.  The Yamabe flow was introduced by Hamilton in 1980s on closed manifolds and its convergence has been established by Chow \cite{Chow}, Ye \cite{Y}, Schwetlick-Struwe \cite{SS} and Brendle \cite{Br05, Br07}. Compared to these results, there are two significant differences in the context of our equation \eqref{eq:main} as follows.

The first one is the zero boundary condition $u=0$ on $\partial\Omega$ in  \eqref{eq:main}, which causes the conformal metric $g = u^{\frac{4}{n-2}} g_0$, associated with the flat metric $g_0$, to become degenerate at the boundary. Nonetheless, this concern can be addressed by the regularity \eqref{eq:reg-jx} that we established earlier. 

The second one is that the solution $u$ of the normalized Yamabe flow  \eqref{eq:main} in our situation blows up, rather than converges. We investigate the dynamics of the bubble in the time-evolving weighted space $L^2(\om,u^{\frac{4}{n-2}}\,\ud x)$, which accommodates the linearization of the flow. This weighted $L^2$ spaces can be decomposed as a direct sum of three subspaces: the one dimensional quasi-unstable space, the $(n+1)$ dimensional quasi-central space, and the quasi-stable space, which correspond to the eigen-space of the negative eigenvalue, the zero eigenvalue, and the positive eigenvalues of the linearized elliptic operator, respectively. Then we study the projections of the flow into these three spaces. We choose an optimal bubble approximation under the weighted $L^2$ norm, which results in that the projection of the flow onto the quasi-stable space is small. The project onto the quasi-unstable space is also small because the normalized Yamabe flow preserves the volume. The projection of the flow onto the $(n+1)$ dimensional quasi-central space is the leading part, and its dynamics is  detected by the $(n+1)$ Pohozaev identities. In those calculations, since the bubble $U_{a,\lda}$  does not satisfy the Dirichlet zero boundary condition on $\partial\Omega$, we need to  apply the projection defined in \eqref{eq:projection}, which introduces additional difficulties due to the error between $U_{a,\lda}$ and $PU_{a,\lda}$.

If the normalized Yamabe flow does not blow up  along a sequence of times, then it converges to a solution of \eqref{eq:elliptic} along this sequence of times, and the uniform relative error convergence in part (i) of Theorem \ref{thm:main} can be proved using {\L}ojasiewicz-Simon's inequality and the regularity \eqref{eq:reg-jx}.

Such classification results for the dynamics of solutions, as stated in Theorem 1, also occur in other parabolic equations. For example, this kind of results for the semilinear heat equation in high dimensional Euclidean spaces has been obtained by Collot-Merle-Rapha\"{e}l \cite{CMR} if the initial data is close to a bubble.

This paper is organized as follows. In Section \ref{sec:Preliminaries}, we provide some preliminaries for the fast diffusion equation and show the the sharp blow-up criterion in  Lebesgue norms in Proposition \ref{prop:integral-crit}. In Section \ref{sec:yamabeflow}, we transform the fast diffusion equation \eqref{eq:rho-1} into the normalized Yamabe flow \eqref{eq:main}, and obtain some properties of the Yamabe flow, including the scalar curvature equation, the concentration compactness, the dichotomy phenomenon  in Theorem \ref{thm:main}, and the relative error convergence if there is no blow up. Section \ref{sec:dynamics} is the main part of this paper, where we investigate the one bubble dynamics of the normalized Yamabe flow, and prove Theorem \ref{thm:main}.

\bigskip

\noindent\textbf{Acknowledgments:} 
We dedicate this paper to the memory of Professor Ha\"im Brezis, whose profound insights greatly influenced us. We are deeply grateful for the generous support he provided to us.

\section{Preliminaries}\label{sec:Preliminaries}

 Let $\rho$  be a nonnegative solution of \eqref{eq:rho-1} and \eqref{eq:rho-2} with $m=\frac{n-2}{n+2}$ and $n\ge 3$, and let $T^*>0$ be its extinction time. It was proved in Berryman-Holland \cite{BH} that,  for every $0<t<T^*$,
 \begin{equation*}
\frac{1}{C}(T^*-t)^{\frac{n}{2}} \le  \int_\om \rho(x,t)^{\frac{2n}{n+2}}\,\ud x \le \left(1-\frac{t}{T^*}\right)^{\frac{n}{2}} \int_\om \rho_0(x)^{\frac{2n}{n+2}}\,\ud x,
 \end{equation*}
 where $C>0$ depends only on $n$. As in \cite{JX20-a},   by the change of  variables 
 \be \label{eq:changing_variables}
 v(x,t):=\left(\frac{n+2}{4(T^*-\tau)}\right)^{\frac{n-2}{4}} \rho^{\frac{n-2}{n+2}}(x,\tau ), \quad t:=\frac{n+2}{4} \ln \left(\frac{T^*}{T^*-\tau}\right), 
 \ee 
we obtain 
 \be \label{eq:BN-3}
 \begin{split}
\frac{\pa }{\pa t}v^{\frac{n+2}{n-2}} &= \Delta v+v^{\frac{n+2}{n-2}} \quad \mbox{in }\om\times (0,\infty),\\
 v&=0 \quad \mbox{on }\pa\om \times (0,\infty), 
\end{split}
\ee
and 
\be \label{eq:BN-5} 
\frac{1}{C} \le \int_{\om} v(x,t)^{\frac{2n}{n-2}}\,\ud x\le C, \quad t\in (0,\infty).
\ee
Denote
 \begin{equation*}
F(t):= \int_{\om } \Big( |\nabla v(x,t)|^2  -\frac{n-2}{n} v(x,t)^{\frac{2n}{n-2}}\Big)\,\ud x. 
 \end{equation*}
By the regularity in \eqref{eq:reg-jx}, we differentiate $F$ in $t$ and find 
\[
\frac{\ud }{\ud t} F(t) =  -\frac{2(n+2)}{n-2} \int_{\om } v^{\frac{4}{n-2}} |\pa_t v|^2 \,\ud x \le 0,  
\]
and 
\begin{align}
\frac{\ud }{\ud t} \int_{\om} v^{\frac{2n}{n-2}} \,\ud x &= \frac{2n}{n+2} \int_{\om}  v\pa_t (v^{\frac{n+2}{n-2}})\,\ud x\nonumber \\
&= \frac{2n}{n+2} \int_{\om} (-|\nabla v|^2 +v^{\frac{2n}{n-2}} ) \,\ud x\nonumber\\&
= - \frac{2n}{n+2} F(t)  + \frac{4}{n+2}  \int_{\om} v^{\frac{2n}{n-2}} \,\ud x . \label{eq:definitionrd}
\end{align}
By Proposition 3.3 of \cite{JX20-a},    $$\lim _{t\to \infty} F(t) =F_\infty >0. $$
It follows that  
\[
\frac{\ud}{\ud t} \left[e^{-\frac{4}{n+2} t}\int_{\om} v^{\frac{2n}{n-2}}\,\ud x\right]= - \frac{2n}{n+2} e^{-\frac{4 t}{n+2} } F(t),
\]
and hence,
\begin{align*}
\int_{\om} v(x,t)^{\frac{2n}{n-2}}\,\ud x &= e^{\frac{4}{n+2} t} \left[e^{-\frac{4}{n+2} }\int_{\om} v(x,1)^{\frac{2n}{n-2}}\,\ud x - \frac{2n}{n+2} \int_1^t  e^{-\frac{4 s}{n+2} } F(s)\, \ud s \right]. 
\end{align*}
Since $\int_{\om} v(x,t)^{\frac{2n}{n-2}}\,\ud x $ is uniformly bounded and $F(t)$ is positive,  it forces that 
\[
e^{-\frac{4}{n+2} } \int_{\om} v(x,1)^{\frac{2n}{n-2}}\,\ud x    - \frac{2n}{n+2} \int_1^\infty  e^{-\frac{4s}{n+2} } F(s)\, \ud s=0. 
\]
Consequently,    
\begin{equation*}
\begin{split}
\int_{\om} v(x,t)^{\frac{2n}{n-2}} \,\ud x &= \frac{2n}{n+2} \int_t^\infty  e^{\frac{4}{n+2} (t-s)} F(s)\, \ud s
=  \frac{2n}{n+2} \int_0^\infty  e^{-\frac{4\tau }{n+2} } F(t+\tau)\, \ud \tau.
\end{split}
\end{equation*}
Plugging this into \eqref{eq:definitionrd}, we obtain
\begin{align}\label{eq:volumederivative}
\frac{\ud }{\ud t} \int_{\om} v(x,t)^{\frac{2n}{n-2}} \,\ud x =- \frac{2n}{n+2} \frac{4}{n+2}\int_0^\infty e^{-\frac{4\tau }{n+2} } [F(t)- F(t+\tau)] \, \ud \tau \le 0 . 
\end{align} 
This, together with \eqref{eq:BN-5}, leads to the first part of the following sharp blow-up criterion in  Lebesgue norms.  

\begin{prop} \label{prop:integral-crit} Let $\rho$  be a nonnegative solution of \eqref{eq:rho-1} and \eqref{eq:rho-2} with $m=\frac{n-2}{n+2}$ and $n\ge 3$, and let $T^*>0$ be its extinction time. Then we have
\[
\lim_{t\to (T^*)^-} (T^*-t)^{-\frac{n+2}{4}}\|\rho(\cdot, t)\|_{L^{\frac{2n}{n+2}}(\om)}=c_1
\]
for some constant $c_1>0$.   Furthermore, if there is no solution of \eqref{eq:elliptic} with $p=\frac{n+2}{n-2}$, then 
\[
\lim_{t\to (T^*)^-} (T^*-t)^{-\frac{n+2}{4}}\|\rho(\cdot, t)\|_{L^{q}(\om)}=\infty, \quad \forall~q>2n/(n+2). 
\]
\end{prop} 

To prove the second part, we need the following proposition. 

\begin{prop} \label{pop:blow-up-crit} If there exist $q>\frac{2n}{n-2}$ and a sequence $t_{j} \to \infty$ as $j\to \infty$ such that 
\[
\limsup_{t_j\to \infty} \|v (\cdot, t_j)\|_{L^q(\om)} <\infty, 
\]
then there exists a subsequence, which is still denoted by $\{t_j\}$,  such that $v(t_j) \to v_\infty $ in $W^{2, \theta}$ for any $\theta\in [1,\infty)$, where $v_\infty $ is a solution of  \eqref{eq:elliptic}. 

\end{prop}  

\begin{proof} First, we claim that for any  $\theta\in [1,\infty)$, 
\[
\sup_{t_j}\|v (\cdot, t_j)\|_{W^{2, \theta }(\om)} <\infty. 
\]

Indeed,  by Corollary 2.7 of \cite{JX20-a}, 
\be \label{eq:cor2.7}
\left\| \mathcal{R}_v-1\right\|_{L^\infty(\om)} \to 0 \quad \mbox{as }t\to \infty,
\ee
where $\mathcal{R}_v:=-v^{-\frac{n+2}{n-2}}\Delta v$. 
Using the assumption of the proposition, $V_j:=\mathcal{R}_{v} \cdot v^{\frac{4}{n-2}} \Big|_{t_j} $ satisfies $\|V_j\|_{L^{\frac{q(n-2)}{4}}} \le C$ for some $C>0$ independent of $j$.  Since $\frac{q(n-2)}{4}>\frac{n}{2}$, applying the De Giorgi-Nash-Moser estimate (see Theorem 8.17 in Gilbarg-Trudinger \cite{GT2001}) to 
\[
-\Delta v(x,t_j)= V_j\cdot  v(x, t_j) \quad \mbox{in }\om, \quad v(\cdot, t_j)=0 \quad \mbox{on }\pa \om, 
\]
we obtain
\[
\|v(\cdot,t_j)\|_{L^\infty(\om)}\le C \|v(\cdot,t_j)\|_{L^{\frac{2n}{n-2}}(\om)}\le C.
\]
This implies that $\|V_j\|_{L^\infty(\om)}\le C$. The claim then follows from the $W^{2,\theta}$ regularity theory of the linear elliptic equations (see Theorem 9.13 in Gilbarg-Trudinger \cite{GT2001}). 

By the above claim, we can find a subsequence, which is still denoted by $t_j$, such that $v(\cdot,t_j) \to v_\infty$ uniformly on $\Omega$, where $v_\infty \ge 0$. Together with \eqref{eq:cor2.7}, we have $\mathcal{R}_{v} \cdot v^{\frac{n+2}{n-2}} \Big|_{t_j}\to v_\infty^{\frac{n+2}{n-2}}$ uniformly on $\Omega$. This implies that
\[
v(\cdot,t_j) \to v_\infty \quad \mbox{in }W^{2,\theta}(\om ) \quad \mbox{as }t_j \to \infty,
\]
and thus,
\[
 -\Delta v_\infty(x)=- \lim_{t_j \to \infty} \Delta v(x,t_j)=   \lim_{t_j \to \infty}  \mathcal{R}_{v} \cdot v^{\frac{n+2}{n-2}} \Big|_{t_j} = v_\infty ^{\frac{n+2}{n-2}}(x)\ \ a.e. \mbox{ in }\om. 
\] By  \eqref{eq:BN-5} and the strong maximum principle,  $v_\infty$ must be positive. The proposition is proved. 
\end{proof} 

\begin{proof}[Proof of Proposition \ref{prop:integral-crit}] 
The first part follows from  \eqref{eq:BN-5} and \eqref{eq:volumederivative}. To prove the second part, we argue by contradiction. Suppose that there exist $q>\frac{2n}{n-2}$ and a sequence $t_j \to \infty$ such that 
\[
\limsup_{t_j \to \infty} \|v(\cdot,t_j)\|_{L^q(\om)} <\infty. 
\]
By Proposition \ref{pop:blow-up-crit}, we find a solution of \eqref{eq:elliptic}.  This contradicts to the nonexistence assumption. Therefore, 
\[
\lim_{t\to \infty} \|v(\cdot,t)\|_{L^q(\om)} =\infty \quad \forall~q>\frac{2n}{n-2}. 
\]
The conclusion follows from the definition of $v$ in \eqref{eq:changing_variables}. 
\end{proof}

\section{Dirichlet problem for the Yamabe  flow}\label{sec:yamabeflow}

We may further normalize $v$ as 
\be \label{eq:from innormal to normal}
u( x,s): = \frac{v(x,t)}{\|v(\cdot, t)\|_{L^{\frac{2n}{n-2}} } }, \quad t=\beta(s) \quad \mbox{with }\beta'(s)=\frac{\ud}{\ud s}\beta(s)= \|v(\cdot,\beta(s))\|_{L^{\frac{2n}{n-2}} }^{\frac{4}{n-2}}, 
\ee
where we dropped $\om$ in $ L^{\frac{2n}{n-2}} (\om )$.  By a direct computation, we obtain the normalized Yamabe flow
\be \label{eq:main}
\begin{split}
\frac{\pa }{\pa t}u^{\frac{n+2}{n-2}} &= \Delta u+r(t) u^{\frac{n+2}{n-2}} \quad  \mbox{in }\om\times (0,\infty),\\
u&=0 \quad\quad\quad\quad \mbox{on }\pa\om \times (0,\infty),\\
\end{split}
\ee
where
\be\label{eq:definitionr}
r (s)= \frac{\int_{\om} |\nabla v|^2\,\ud x }{\|v\|_{L^{\frac{2n}{n-2}} } ^{2}}=   \frac{\int_{\om} |\nabla u|^2\,\ud x}{ \int_{\om }u^{\frac{2n}{n-2}}\,\ud x} .
\ee
Conversely, one can also transform this normalized Yamabe flow to the fast diffusion equation \eqref{eq:rho-1}.

From now on we consider the Yamabe flow with the homogenous boundary condition. We may always assume that $u$ is normalized from $v$ in \eqref{eq:from innormal to normal} so that we can directly use the results in \cite{JX19,JX20-a} for $v$. In fact, one can still prove all the results below without using $v$.

Let $g=u^{\frac{4}{n-2}} g_0$ and  $R_g$ be the scalar curvature of $g$, where $g_0$ is the flat metric. Set 
\[
M_q(t)= \left.\int_{\om }|\mathcal{R}-r|^q u^{\frac{2n}{n-2}}\,\ud x\right|_{t}, 
\]
where 
\[
\mathcal{R}:=-u^{-\frac{n+2}{n-2}}\Delta u=\frac{n-2}{4(n-1)} R_g.
\]   
Hence, the conformal transformation law  of the conformal Laplacian leads to 
\be \label{eq:conf-law}
(\Delta_g -\mathcal{R})\phi= u^{-\frac{n+2}{n-2}} \Delta (u\phi), \quad  \forall~\phi\in C^2(\om),
\ee
where $\Delta_g$ is the Laplace-Beltrami operator with respect to $g$. By \eqref{eq:main}, we also have
\[
\mathcal{R}=r(t)-\frac{n+2}{n-2}\frac{u_t}{u}.
\] 
Using the regularity in \eqref{eq:reg-jx}, one has
\[
\pa_t^l  u \in C^{2+\frac{n+2}{n-2}} (\overline \om \times (0,\infty)), \quad \forall~ l\ge 0. 
\]
This implies that $\mathcal{R}(\cdot, t) \in C^{1+\frac{n+2}{n-2}}(\overline \om) $ for all $t>0$.   Moreover, it follows from DiBenedetto-Kwong-Vespri \cite{DKV} that for any $0<t_1<t_2<\infty$ one can find $C>0 $ depending on $t_1,t_2$ and $u$ such that 
\begin{equation*}
1/C \le u/d\le C \quad \mbox{on }\overline \om \times [t_1,t_2],
\end{equation*}
where $d(x):=dist(x, \pa \om)$.  As a result, we have the  integration by parts formula (Lemma 2.2 of \cite{JX20-a})  
\be\label{eq:integrationbyparts}
\int_{\om}h \Delta_g f\,\ud vol_g = -\int_{\om} \langle \nabla_g f, \nabla_g h\rangle_g \,\ud vol_g, \quad \forall ~f\in W^{2,2}(\om), ~h\in  W^{1,2}(\om),
\ee
where $\nabla_g$ the gradient vector field with respect to $g$. Note that there is no boundary term appeared in \eqref{eq:integrationbyparts}.    

\begin{lem} \label{lem:ppty} We have the following facts. 
\begin{itemize}
\item[(i)] There holds
\[
\pa_t u^{\frac{2n}{n-2}}= -\frac{2n}{n+2}(\mathcal{R}-r) u^{\frac{2n}{n-2}}. 
\]
Consequently, the volume $Vol_{g}(\om):= \int_{\om} u^{\frac{2n}{n-2}} \,\ud x$ is preserved.  Without loss of generality, we assume $Vol_{g}(\om)=1$. 
\item[(ii)] There holds
\begin{align*}
 \pa_t(\mathcal{R}-r)=\frac{n-2}{n+2}\Delta_g (\mathcal{R}-r) +\frac{4}{n+2} \mathcal{R}(\mathcal{R}-r)-\dot{r},
\end{align*}
where $\dot{r}(t)=\frac{\ud}{\ud t}r(t)$. 
\item[(iii)] For the functional $Y_\Omega(u)$ defined in \eqref{eq:energyfunctional}, there holds $$\frac{\ud }{\ud t}Y_\Omega(u)=\dot{r}=-\frac{2(n-2)}{n+2} M_2\le 0.$$ Hence, the limit $r_\infty:=\lim_{t\to \infty} r(t)$ exists and $r_\infty\ge K(n)$.
\end{itemize}

\end{lem}

\begin{proof} The first item follows immediately from the definition of $\mathcal{R}$ that
\be\label{eq:uteq}
\frac{\pa _t u}{u}=\frac{n-2}{n+2}(-\mathcal{R}+r).
\ee
The second item follows from the computations that 
\begin{align*}
 \pa_t\mathcal{R}&=\frac{n+2}{n-2} u^{-\frac{n+2}{n-2}} \cdot \frac{\pa_t u}{u}\cdot \Delta u -u^{-\frac{n+2}{n-2}} \cdot \Delta (u \cdot \frac{\pa_t u}{u})\\&
 =-\mathcal{R} (-\mathcal{R}+r) -(\Delta_g- \mathcal{R})  (\frac{\pa_t u}{u})\\&
 = \frac{n-2}{n+2}\Delta_g (\mathcal{R}-r) +\frac{4}{n+2}\mathcal{R}(\mathcal{R}-r),
\end{align*}
where we used the conformal transformation law \eqref{eq:conf-law} in the second equality. 

The third item follows from the calculations that
\begin{align*}
\frac{\ud }{\ud t} Y_\Omega(u)&=2\int_{\om} \nabla u \nabla \pa_tu\,\ud x\\&
=-2   \int_{\om} \Delta u \cdot \frac{\pa _t u}{u}\cdot u\,\ud x\\&
= \frac{2(n-2)}{n+2}\int_{\om} \mathcal{R}(-\mathcal{R}+r)  u^{\frac{2n}{n-2}}\,\ud x\\&
= -\frac{2(n-2)}{n+2}\int_{\om} (\mathcal{R}-r)^2   u^{\frac{2n}{n-2}}\,\ud x.
\end{align*}
The lemma is proved.
\end{proof}

\begin{prop} \label{prop:jx20} We further  have
\begin{itemize}
\item[(i)] $\lim_{t\to \infty} \|\mathcal{R}-r_\infty\|_{L^\infty(\om)}=0$.
\item[(ii)] For any $t_\nu \to \infty$, $\nu \to \infty$, $u_\nu=u(\cdot, t_\nu)$ is a Palais-Smale sequence of the functional $Y_\Omega$ in $H_0^1(\om)$.
\item[(iii)] There exist two positive constants  $\delta_0$ and $C$, depending on $u(\cdot,1)$,   such that,
\[
u\le C d \quad \mbox{in }  \{y\in \om: d(y)<\delta_0\} \times [1,\infty). 
\]
\end{itemize}
\end{prop} 

\begin{proof} Using the change of variables \eqref{eq:from innormal to normal}, the proposition follows from  Corollary 2.7,  Proposition 3.1 and Proposition 3.2 of \cite{JX20-a} for $v$ satisfying \eqref{eq:BN-3} and \eqref{eq:BN-5}. 
\end{proof}

Recall the functions $U_{a,\lda}$ defined in \eqref{eq:U} which satisfies \eqref{eq:Uequation}, and  $PU_{a,\lda}=U_{a,\lda}  -h_{a,\lda}$, the projection of $U_{a,\lda}$ into $H^1_0(\Omega)$, defined in \eqref{eq:projection}. 
 By the maximum principle, $PU_{a,\lda}$ and $h_{a,\lda}$ are positive in  $\Omega$.   

\begin{prop} \label{prop:s-b-c} 
For any $t_\nu \to \infty$, $\nu \to \infty$,  after passing to a subsequence if necessary, $u_\nu:=u(\cdot, t_\nu)$ weakly converges to $u_\infty$ in $H_0^1(\om)$, where $u_\infty$ is nonnegative and satisfies $-\Delta u_\infty= r_\infty u_\infty^{\frac{n+2}{n-2}} $ in $\om$.  

Moreover, we can find a non-negative integer $\ell$ and a sequence of $\ell$-tuplets $(a^*_{k,\nu}, \lda_{k,\nu}^*)_{1\le k\le \ell}$ with $(a^*_{k,\nu}, \lda_{k,\nu}^*) \in \om \times (0,\infty)$ satisfying 
$\lim_{\nu\to \infty} \lda_{k, \nu}^* = \infty$, $d(a^*_{k,\nu}) >\delta_0/2$,  and for each pair $k\neq l$,
\be\label{eq:seperate}
\frac{ \lda_{k, \nu}^* }{\lda_{l, \nu}^* }+\frac{ \lda_{l, \nu}^* }{\lda_{k, \nu}^* }+\lda_{k, \nu}^* \lda_{l, \nu}^* |a^*_{k,\nu}-a^*_{l,\nu} |^2 \to \infty,
\ee
 where $\delta_0>0$ is the constant in  Proposition \ref{prop:jx20}, such that 
\be\label{eq:seperate2}
\lim_{\nu \to \infty }\left\|u_\nu- u_\infty- r_\infty^{-\frac{n-2}{4}}\sum_{k=1}^\ell PU_{a^*_{k,\nu}, \lda_{k,\nu}^*}\right\|_{H^1_0(\Omega)} = 0. 
\ee
Consequently, 
\[
r_\infty = (Y_\Omega(u_\infty) ^{\frac{n}{2}} +\ell K(n)^{\frac{n}{2}})^{\frac{2}{n}},
\]
where we set $Y_\Omega(u_\infty)  =0$ if $u_\infty \equiv 0$.

\end{prop} 

\begin{proof} Given Proposition \ref{prop:jx20},  this compactness statement is standard by now; see Struwe \cite{St},   Br\'ezis-Coron \cite{BC} and Bahri-Coron \cite{Bahri-C}.   Due to item (iii) of Proposition \ref{prop:jx20},  the bubbles' centers must uniformly stay away from the boundary, i.e, $d(a^*_{k,\nu}) >\delta_0/2$,   

It remains to quantify $r_\infty$. By the strong maximum principle,   we have that either   $u_\infty>0$ in $\om$ or $u_\infty\equiv 0$. By using \eqref{eq:seperate}, \eqref{eq:seperate2} and the inequality
\[
(a+b)^{\frac{2n}{n-2}}-a^{\frac{2n}{n-2}}-b^{\frac{2n}{n-2}}\le Ca^{\frac{n-2}{n-2}}b+Cab^{\frac{2n}{n-2}}\quad\forall~a,b\ge 0,
\]
we have
\begin{align*}
\lim_{t_\nu \to \infty} \int_{\om} u_\nu^{\frac{2n}{n-2}}\,\ud x&=  \lim_{t_\nu \to \infty} \int_{\om} \Big(u_\infty^{\frac{2n}{n-2}}+  r_\infty^{-\frac{n}{2}} \sum _{k=1}^\ell  (PU_{a^*_{k,\nu} })^{\frac{2n}{n-2}} \Big)\,\ud x.
\end{align*}
Hence,
\[
1=\left(\frac{Y_\Omega(u_\infty)}{r_\infty}\right)^{\frac{n}{2}} + \ell \left(\frac{K(n) }{r_\infty}\right)^{\frac{n}{2}},
\]
where we used the equation of $u_\infty$ and the definition of $Y_\Omega(u_\infty)$ to obtain 
\[
r_\infty \left( \int_{\om} u_\infty^{\frac{2n}{n-2}}\right)^{\frac{2}{n}} = Y_\Omega(u_\infty)
\]
and used 
\[
\lim_{\nu \to \infty} \int _{\om }(PU_{a^*_{k,\nu} })^{\frac{2n}{n-2}} = K(n)^{\frac{n}{2}}.
\]
It follows that $
r_\infty = (Y_\Omega(u_\infty)^{\frac{n}{2}} +\ell K(n)^{\frac{n}{2}})^{\frac{2}{n}}.$ The proposition is proved. 
\end{proof}

\begin{prop}\label{prop:subseq-converge} 
If there exists a sequence $t_\nu \to \infty$, $\nu \to \infty$,  such that $u(\cdot, t_\nu)$ converges to $u_\infty$ in $L^{\frac{2n}{n-2}}(\om)$ for some positive $C^2$ function  $u_\infty$ satisfying $-\Delta u_\infty= r_\infty u_\infty^{\frac{n+2}{n-2}} $ in $\om$ and $u_\infty=0$ on $\pa \om$,  then 
\[
\left\|\frac{u(\cdot, t)}{u_\infty}-1\right\|_{C^2(\overline \om)} \le C t^{-\gamma} \quad \forall ~t>1,
\]
where  $C$ and $\gamma$ are positive constants. 
\end{prop}

\begin{proof}   
The idea of the proof is due to Simon \cite{Simon}.  

We note that, for any $1<a<b<\infty$, 
\begingroup
\allowdisplaybreaks
\begin{align}  
\left(\int_{\om} |u(x,b)-u(x,a)|^{\frac{2n}{n-2}}\,\ud x\right)^{1/2} &\le\left( \int_{\om} |u(x,b)^{\frac{n}{n-2}} -u(x,a) ^{\frac{n}{n-2}} |^2\,\ud x\right)^{1/2} \nonumber\\&
=\left(\int_{\om} \left|\int_{a}^{b} \pa_t u(x,t)^{\frac{n}{n-2}} \,\ud t\right| ^2 \,\ud x \right)^{1/2}\nonumber\\&
\le \int_{a}^{b} \left(\int_{\om} \left| \pa_t u(x,t)^{\frac{n}{n-2}}\right| ^2 \,\ud x\right)^{1/2}\,\ud t\nonumber\\&
=\frac{n}{n+2} \int_{a}^{b} \left( \int_{\om} |\mathcal{R}(x,t)-r(t)|^2 u^{\frac{2n}{n-2}} \,\ud x \right)^{1/2} \,\ud t\nonumber\\&
= \frac{n}{n+2} \int_{a}^{b} M_2(t)^{1/2} \,\ud t.\label{eq:stable-1}
\end{align}
\endgroup

Hence, our goal is to prove $M_2(t)^{1/2} \in L^1(1,\infty)$. 

\textit{Step 1.} Set up the framework.

\textit{(1.a)}  Let $C_0>0$ be a constant  such that 
\[
\frac{1}{C_0}\le \frac{u_\infty}{d }\le C_0 \quad \mbox{in }\om
\]
and then set 
\[
\Xi_\sigma: =\left\{w\in W^{2,n+1}_0(\om): \frac{1}{\sigma C_0}< \frac{w}{d }< \sigma C_0 \quad \mbox{in }\om\right\}, \quad \sigma\ge 1. 
\]

\textit{(1.b)} By Feireisl-Simondon \cite{FS}, $Y_\Omega$ is analytic in $\Xi_{4}$. Moreover,  there exist $\va_0>0$,  $\theta\in (0,1)$ and $C_1>0$ such that 
\be \label{eq:LS}
|Y_\Omega(w)-Y_\Omega(u_\infty)| \le C_1\| D Y_\Omega(w)\|^{1+\theta}_{H^{-1}(\Omega)}   
\ee
for all $w\in \Xi_{4}$ satisfying $
 \|w-u_\infty\|_{L^{\frac{2n}{n-2}}(\om)} \le \va_0$, where $D Y_\Omega(w)$ is the Frech\'et differential of $Y_\Omega$ at $w$.  

\textit{(1.c)} Since $\lim_{t\to \infty} \|\mathcal{R}-r_\infty\|_{L^\infty(\om)}=0$,  there exist $\va_1>0$ and $T_0>1$ such that whenever $\|u(\cdot,t)-u_\infty\|_{L^{\frac{2n}{n-2}}(\om)}\le \va_1$ with $t\ge T_0$, then $u(\cdot,t)\in L^\infty(\Omega)$, and thus,  $u(\cdot,t)\in\Xi_2$ and 
\[
 \|u(\cdot,t)-u_\infty\|_{W^{2, n+1}(\om)} \le \va_0/2.
\]
This can be proved by using Moser's iteration with incorporating ideas of Br\'ezis-Kato \cite{BrezisKato},  together with higher order regularity theory for linear elliptic equations.  See page 1321-1322 of \cite{JX20-a} for more details. 

\textit{(1.d)}   Without loss of generality, we may assume $t_\nu$ is an increasing sequence. Since $M_2^{1/2}\le \|\mathcal{R}-r_\infty\|_{L^\infty(\om)}$  and $\lim_{\nu\to \infty} \|u(\cdot,t_\nu)-u_\infty\|_{L^{\frac{2n}{n-2}}}=0$, by using \eqref{eq:stable-1}, there exists $\nu_0>1$ with $t_{\nu_0}>T_0$ such that 
\[
\bar \mu_\nu := \sup\Big\{\mu> t_{\nu_0}:  \|u(\cdot, t)-u_\infty\|_{L^{\frac{2n}{n-2}}} <\va_1/2, \quad\forall~ t\in [t_{\nu}, \mu )\Big \}
\]
is well defined for all $\nu\ge \nu_0$ and 
\be \label{eq:mu-large}
\lim_{\nu \to \infty}{\bar \mu}_{\nu}=\infty. 
 \ee

\textit{Step 2.} We claim that $\bar \mu_{\nu_1}=\infty$ for some $\nu_1\ge \nu_0$. 

Let $L>0$ be a large number to be fixed. By  \eqref{eq:mu-large} and arguing as in \textit{(1.d)}, we can find $\nu_1\ge\nu_0$ such that $t_{\nu_1}+2^3 L <\bar \mu_{\nu_1} $ and 
\be \label{eq:press-1}
 \|u(\cdot,t)-u_\infty\|_{L^{\frac{2n}{n-2}}} <\va_1/10 \quad \forall~ t\in [t_{\nu_1}, t_{\nu_1}+2^3 L]. 
\ee  
We may assume $\bar \mu_{\nu_1}<\infty$, otherwise we are done.  Let $J_0$ be the largest number so that $ t_{\nu_1}+L2^{J_0}<\bar \mu_{\nu_1}$. 
Let $a_l= t_{\nu_1}+L2^{l}$ for $l=1,\dots, J_0$ and $a_{J_0+1}= \bar \mu_{\nu_1}$. 

For  $ t\in (t_{\nu_1}, \bar \mu_{\nu_1})$, by Step 1, \eqref{eq:LS} and the monotonicity of $Y_\Omega(u(\cdot,t))$ we have 
\[
0\le Y_\Omega(u(\cdot,t))-Y_\Omega(u_\infty) \le C_1\| D Y_\Omega(u(\cdot,t))\|^{1+\theta}_{H^{-1}(\Omega)} \le K(n)^{-1/2} C_1 M_2(t)^{\frac{1+\theta}{2}}. 
\]
By item (iii) of Lemma \ref{lem:ppty}, 
\begin{align*}
\frac{\ud }{\ud t}(Y_\Omega(u(\cdot,t))-Y_\Omega(u_\infty)  )&= -\frac{2(n-2)}{n+2} M_2(t) \\&\le - \frac{2(n-2)}{n+2}  (K(n)^{-1/2} C_1)^{-\frac{2}{1+\theta}  }  (Y_\Omega(u(\cdot,t))-Y_\Omega(u_\infty))^{\frac{2}{1+\theta}}. 
\end{align*}
That is
\[
\frac{\ud }{\ud t}(Y_\Omega(u(\cdot,t))-Y_\Omega(u_\infty)  )^{\frac{\theta-1}{\theta+1}} \ge \frac{2(n-2)}{n+2}  (K(n)^{-1/2} C_1)^{-\frac{2}{1+\theta}  }   \frac{1-\theta}{1+\theta}=:c>0. 
\] 
It follows that for any $l=2,\dots, J_0+1$,  
\[
Y_\Omega(u(\cdot,a_l))-Y_\Omega(u_\infty)   \le c^{-\frac{1+\theta}{1-\theta}}(a_l-a_{l-1})^{-\frac{1+\theta}{1-\theta}}.
\]
Using H\"older's inequality and  item (iii) of Lemma \ref{lem:ppty} again, for $l\ge 3$, 
\begin{align*}
\left(\int_{a_{l-1}}^{a_{l}} M_2(s)^{1/2}\,\ud s\right)^2&  \le (a_l-a_{l-1})  \int_{a_{l-1} }^{a_l } M_2(s) \,\ud s  
\\& \le \frac{n+2}{2(n-2)}  (a_l-a_{l-1})   (Y_\Omega(u(\cdot,a_{l-1}))- Y_\Omega(u(\cdot,a_l))) \\&
\le  \frac{n+2}{2(n-2)} (a_l-a_{l-1})   (Y_\Omega(u(\cdot,a_{l-1}))- Y_\Omega(u_\infty))   \\& 
\le \frac{n+2}{2(n-2)}  (a_l-a_{l-1})  c^{-\frac{1+\theta}{1-\theta}}(a_{l-1}-a_{l-2})^{-\frac{1+\theta}{1-\theta}}   \\&
 \le C L^{-\frac{2\theta}{1-\theta}} (2^{- \frac{2\theta}{1-\theta}})^{l-1},
\end{align*} 
where $C>0$ depends only on $n,C_1 $ and $ \theta$. 
Hence,
\begin{align*}
\int_{\nu_1+L 2^{2}}^{\bar \mu_{\nu_1}}  M_2(s)^{1/2}\,\ud s \le L^{-\frac{\theta}{1-\theta}}  \cdot \frac{C}{2^{\frac{\theta}{1-\theta}} -1}. 
\end{align*}
Choose $L$ sufficiently large such that 
\[
 L^{-\frac{\theta}{1-\theta}}  \cdot \frac{C}{2^{\frac{\theta}{1-\theta}} -1} <\Big(\frac{\va_1}{10} \Big)^{\frac{n}{n-2}}.
\] 
It is clear that  $L$ depends only on $n,C_1 $, $ \theta$ and $\va_1$. By  \eqref{eq:press-1} and \eqref{eq:stable-1}, we have for any $t\in [t_\nu+2^2L, t_\nu+ \bar \mu_{\nu_1}]$ that 
\begin{align*}
\|u(\cdot,t)-u_\infty\|_{L^{\frac{2n}{n-2}}(\om)} &\le \|u(\cdot,t)-u(\cdot,t_\nu+2^2L)\|_{L^{\frac{2n}{n-2}}(\om)} +\|u(\cdot,t_\nu+2^2L)-u_\infty\|_{L^{\frac{2n}{n-2}}(\om)} \\&
\le \frac{\va_1}{10} +\frac{\va_1}{10} = \frac{\va_1}{5}. 
\end{align*}
In particular, $\|u(\cdot,t_\nu+ \bar \mu_{\nu_1})-u_\infty\|_{L^{\frac{2n}{n-2}}(\om)}\le \va_1/5<\va_1/2$.  Since $\|u(\cdot,t)-u_\infty\|_{L^{\frac{2n}{n-2}}(\om)}$ is continuous in $t$, we obtained a contradiction to the defintion of $\bar \mu_{\nu_1}$.  Hence, $\bar \mu_{\nu_1}=\infty$ and the above argument still leads to 
\begin{align*}
\int_1^\infty  M_2(s)^{1/2}\,\ud s &= \int_1^{\nu_1+L 2^{2}} M_2(s)^{1/2}\,\ud s+  \int_{\nu_1+L 2^{2}}^{\infty}  M_2(s)^{1/2}\,\ud s\\&
<\int_1^{\nu_1+L 2^{2}} M_2(s)^{1/2}\,\ud s +\Big(\frac{\va_1}{10} \Big)^{\frac{n}{n-2}}\\
&< \infty. 
\end{align*}
Therefore, $\lim_{t\to \infty} \|u(\cdot,t)-u_\infty\|_{L^{\frac{2n}{n-2}}(\om)}=0$ and $u(\cdot,t)\in \Xi_2$ for $t\ge t_{\nu_1}$.  

The proof of the  convergence in the $C^2$ topology with a polynomial rate is identical to the proof in Section 5 of \cite{JX20-a}.    We omit the details here. The proposition is proved. 
\end{proof}

\begin{cor}\label{cor:alternative} Let $u_\infty$ and $\ell$ be those obtained in Proposition \ref{prop:s-b-c}.  If $Y_\om (u(\cdot, 0)) \le 2^{\frac{2}{n}}  K(n)$, then either
\begin{itemize} 
\item[(i)] $u_\infty >0$ in $\om$ and $\ell=0$, or 
\item[(ii)] $u_\infty \equiv 0$ in $\om$ and $\ell=1$. 
\end{itemize}

Furthermore, if  item (i) happens, then there exist positive constants $\gamma$ and $C$ such that 
\be\label{eq:uinfinityconvergence}
\left\|\frac{u(\cdot,t)}{u_\infty}-1\right\|_{C^2(\overline \om)} \le C t^{-\gamma},  \quad \forall ~t>1. 
\ee
\end{cor}

\begin{proof}  By item (iii) of Lemma \ref{lem:ppty}, $Y_\Omega(u(\cdot,t))=r(t) > K(n)$ is non-increasing, and hence, 
\[
(Y_\Omega(u_\infty) ^{\frac{n}{2}} +\ell K(n)^{\frac{n}{2}})^{\frac{2}{n}} \le 2^{\frac{2}{n}}   K(n). 
\]
If $u_\infty>0$, then $Y_\Omega(u_\infty) >K(n)$ as the best constant of the Sobolev inequality can never be achieved in bounded domains. This forces that $\ell=0$.  

If $u_\infty\equiv 0 $, then $\ell\in \{1,2\}$.   If $\ell=2$, then 
\[
Y_\Omega(u(\cdot,t))=2^{\frac{2}{n}}   K(n)
\]
 for all $t\ge 0$. Hence, 
 \[
 0\equiv \frac{\ud }{\ud t}Y_\Omega(u)=-\frac{2(n-2)}{n+2} M_2.
 \] 
 Therefore, $\mathcal{R}\equiv 2^{\frac{2}{n}}  K(n)$ and thus $u$ is independent of $t$, which contradicts to $u_\infty\equiv 0$. Therefore, $\ell=1$.  

If the item (i) happens, by Proposition \ref{prop:s-b-c} there exists a sequence $t_i \to \infty$, $i\to \infty$,  such that 
$u(\cdot,t_i)$ strongly converge to $u_\infty$ in $H_0^1$.  The conclusion follows from Proposition \ref{prop:subseq-converge}. The corollary is proved. 
\end{proof}

\section{One bubble dynamics}\label{sec:dynamics}

In the rest of the paper, we study the one bubble dynamics, that is, we assume that
\[
u_\infty\equiv 0\quad\mbox{and} \quad \ell=1
\]
for any sequence of times in Proposition \ref{prop:s-b-c}. By Corollary \ref{cor:alternative}, this assumption is fulfilled if $Y_\om (u(\cdot, 0)) \le 2^{\frac{2}{n}}   K(n)$ and there is no solution of \eqref{eq:elliptic} with $p=\frac{n+2}{n-2}$. Then, it follows that 
\[
r_\infty=K(n).
\] 

\subsection{Choice of an optimal bubble approximation}
Let us recall the neighborhood of critical points at infinity of Bahri \cite{Bahri}.  
 For $\va>0$,  define a tuplets $(a,\lda, \al)$ with $a\in \om$ and $\lda,\al\in (0,\infty)$ as 
\begin{align*}
A_{u(t)}(\va)=\Big \{(a,\lda, \al): ~& d(a)>\frac{\delta_0}{2},~ \frac{1}{\lda}<\va, ~|\al-1|<\va,
\left\|u(\cdot,t)-r_\infty^{-\frac{n-2}{4}} \al PU_{a,\lda} \right\|_{H_0^1(\om)}<\va \Big\},
\end{align*}
where $\delta_0$ is the constant in  Proposition \ref{prop:jx20}. It follows from  Proposition  \ref{prop:s-b-c} that for any $\va>0$, there exists a constant $\overline  T_0(\va)>1$ such that    
\begin{align*}
A_{u(t)}(\va)\neq \emptyset , \quad \mbox{for all }t\ge \overline  T_0(\va). 
\end{align*}

To accommodate the linearized operator $\frac{n+2}{n-2} \pa_t - u^{-\frac{4}{n-2}}\Delta-\frac{n+2}{n-2}r$, 
we introduce  the weighted space $ \mathcal{L}_t^2:= L^2(\om,u^{\frac{4}{n-2}}\,\ud x)$ with the inner product
\begin{align}\label{eq:innerprod}
\langle f, g\rangle_{\mathcal{L}_t^2 }= \int_{\om} fgu^{\frac{4}{n-2}}\,\ud x.
\end{align}
By adapting the proof of Proposition 0.7 of Bahri \cite{Bahri} or Proposition 3.10 of Mayer \cite{Mayer} to our context,   one can show
\begin{prop}\label{prop:optimal choice} There exists a small $\va_0>0$ such that for every $\va\in (0,\va_0)$, one can find  $ \overline  T_1(\va)>\overline T_0(\va)$ such that, for $t\ge  \overline  T_1(\va)$,  the variational problem 
\be \label{eq:opt-l2}
\inf _{(a^*,\,\lda^*,\,\al^*) \in A_{u(t)}(\va)} \left \|u- r_\infty^{-\frac{n-2}{4}}  \al^*PU_{a^*,\lda^*}\right \|_{\mathcal{L}_t^2 }
\ee
has a unique minimizer $(a,\lda, \al)\in A_{u(t)}(\va)$ that smoothly depends on $t\in [\overline {T_1}(\va), \infty)$. Moveover, $\mbox{dist}(a,\partial\Omega)>\delta_0/2$, where $\delta_0$ is the constant in  Proposition \ref{prop:jx20}. 
 \end{prop}
Let 
\[
 w=u- r_\infty^{-\frac{n-2}{4}} \alpha  PU_{a, \lda} .
\]
 
\subsection{Spectrum analysis for the linearized operator}
 
Note that 
\begin{align*}
\lda^{\frac{n-2}{2}} U_{a, \lda}(x)&= [n(n-2)]^{\frac{n-2}{4}} \Big(\frac{1}{\lda^{-2} +|x-a|^2}\Big)^{\frac{n-2}{2}} \\& \to  [n(n-2)]^{\frac{n-2}{4}} |a-x|^{2-n} \quad \mbox{in }C_{loc}^2(\R^n\setminus \{a\})\quad\mbox{as }\lda\to \infty. 
\end{align*}
For $a\in \om$, we let  $H(a,x)$ be the solution of 
\[
\begin{cases}
-\Delta H(a, x)=0, \quad x\in \om, \\[2mm] H(a, x)=[n(n-2)]^{\frac{n-2}{4}} |a-x|^{2-n}, \quad x\in \pa \om. 
\end{cases}
\]

\begin{prop}[Proposition 1 of Rey \cite{Rey}] \label{prop:bubble-derivative}  
Suppose $a\in \om$ with $\mbox{dist}(a,\partial\Omega)\ge \delta>0$ and $\lda>1$.  Let $H(a, \cdot)$ be defined as above and $h_{a,\lda}$ satisfy \eqref{eq:h-a}. Then 
\[
f_{a,\lda} (x):= h_{a,\lda}(x) - \lda^{-\frac{n-2}{2}} H(a,x), \quad  x\in \overline\om
\]
is a smooth function on $\overline \om$, also smooth in parameters $a$ and $\lda$. Moreover,  there hold 
\begin{align*}
& |f_{a,\lda} | \le C \lda^{-\frac{n+2}{2}}, \quad |\pa_{a^j} f_{a,\lda}| \le  C \lda^{-\frac{n+2}{2}}, \quad |\pa_\lda f_{a,\lda} | \le  C\lda^{-\frac{n+4}{2}}\\
& |\pa_{a^j}  \pa_\lda f_{a,\lda} | \le  C\lda^{-\frac{n+4}{2}},\quad |\pa_{a^j} \pa_{a^k}   f_{a,\lda} | \le  C\lda^{-\frac{n+4}{2}},\quad |\pa_\lda^2 f_{a,\lda} | \le  C\lda^{-\frac{n+6}{2}}, 
\end{align*}
for every $j, k=1,\dots, n, $ where  $C$ depends only on $n, \om$ and $\delta$.
\end{prop}
\begin{proof}
The first order derivative estimates are in Proposition 1 of Rey \cite{Rey}. The second order derivative estimates can be obtained similarly, using standard elliptic estimates for the Laplace equation.
\end{proof}

It was also proved in Rey \cite{Rey} that 
\begin{align*}
\frac{1}{\lda}\pa_{a^j} U_{a,\lda} (x)&=(n-2) U_{a,\lda} \frac{\lda(x_j-a^j)}{1+\lda^2|x-a|^2},   \quad j=1,\dots, n, \\
\lda \pa_{\lda} U_{a,\lda} (x)&=\frac{(n-2)}{2} U_{a,\lda} \frac{1-\lda^2|x-a|^2}{1+\lda^2 |x-a|^2} ,
\end{align*}
is a basis of the kernel of the  Jacobi operator 
\[
-\Delta-\frac{n+2}{n-2}   U_{a,\lda} ^{\frac{4}{n-2}} \quad \mbox{in } \R^n.
\]  
This inspires us to introduce 
\begin{align*}
X_{0}&= r_\infty^{-\frac{n-2}{4}}   PU_{a, \lda}, \\
X_{j}&=r_\infty^{-\frac{n-2}{4}}  \frac{1}{\lda} \pa_{a^{j}}  PU_{a, \lda}, \quad j=1,\dots, n, \\ 
X_{n+1}&=r_\infty^{-\frac{n-2}{4}} \lda \pa_{\lda}   PU_{a, \lda}. 
\end{align*}
Then  $$w=u-\al X_0.$$ For brevity, we let 
 \[
  \overline X_0 = r_\infty^{-\frac{n-2}{4}}   U_{a, \lda}.
 \]
Then by Proposition \ref{prop:bubble-derivative}, we have 
\be\label{eq:almost-central}
\begin{split}
-\Delta X_0&=-\Delta \overline X_0= r_\infty \overline X_0^{\frac{n+2}{n-2}},\\ 
-\Delta X_{j}&=r_\infty  \frac{n+2}{n-2}   \overline X_0 ^{\frac{4}{n-2}}  (X_j +O(\lda^{\frac{2-n}2 })), \quad j=1,\dots, n+1.
\end{split}
\ee
\begin{prop}
One has
\be \label{eq:almost-orth-basis}
\begin{split}
\int_\Omega |X_i|^2 u^{\frac{4}{n-2}}&=\alpha^{\frac{4}{n-2}} \kappa_i +O( \lda^{2-n} +\|w\|), \quad i=1,\dots, n+1, \\
\int_\Omega X_i X_j u^{\frac{4}{n-2}}&= O(\lda^{2-n} +\|w\|)  \quad \mbox{for }j=0,\dots,n+1,\ j\neq i, 
\end{split}
\ee
where $\|w\|:= \|w\|_{H^1_0(\Omega)}$, and $ \kappa_i $ are positive constants depending only on $n$.
\end{prop}
\begin{proof}
It follows from Proposition \ref{prop:bubble-derivative} and elementary calculations that are similar to the proof of (B.6) -- (B.9) in Rey \cite{Rey}. In fact, one can calculate that
\begin{align*}
\kappa_1&=\dots=\kappa_n=K(n)^{-\frac{n}{2}} (n-2)^2 \int_{\R^n} U_{0,1}^\frac{2n}{n-2} \frac{|x|^2}{n(1+|x|^2)^2}\,\ud x,\\
\kappa_{n+1}&= K(n)^{-\frac{n}{2}} \frac{(n-2)^2}{4} \int_{\R^n} U_{0,1}^\frac{2n}{n-2} \frac{(1-|x|^2)^2}{(1+|x|^2)^2}\,\ud x,
\end{align*}
where we used $r_{\infty}=K(n)$.
\end{proof}

Set
\begin{align*}
E^{ut}&=span\{X_{0}\}, \\ 
E^{c}&=  span\{X_{j}: j=1,\dots, n+1\},\\
E^s&= (E^{ut} \oplus E^c)^\perp \cap H_0^1(\om),
\end{align*}
where $\perp$ is with respect to $\langle \cdot, \cdot \rangle_{\Lt}$.
We call $E^{ut}$ the quasi-unstable space, $E^{c}$ the quasi-central space, and  $E^{s}$ the quasi-stable space, respectively. By the minimality of \eqref{eq:opt-l2}, 
 \[
 w=u-\al X_0\in E^s.
 \]
 
 \subsection{Estimates of the errors}
 
There are five small quantities as follows: 
\[
h_{a,\lda} \sim \lda^{\frac{2-n}{2}}, \quad M_2, \quad  w,\quad  r(t)-r_\infty, \quad \al-1 .
\]
In this subsection, we shall use the first two to  bound the last three.  

\begin{prop}\label{lem:positivity} 
There exist $c\in (0,1)$ and $t_0\gg1$ such that for all $t\ge t_0$, we have
  \[
 (1-c)\int_{\om} |\nabla f(x)|^2\,\ud x \ge  \frac{r_\infty(n+2)}{n-2}  \int_{\om}   f(x)^2 (\al  X_0)^{\frac{4}{n-2}} \,\ud x, 
 \]
 \[
 (1-c)\int_{\om} |\nabla f(x)|^2\,\ud x \ge  \frac{r_\infty(n+2)}{n-2}  \int_{\om}  f(x)^2  u (x,t)^{\frac{4}{n-2}}\,\ud x,
 \]
for any $f\in E^s$. 
\end{prop}

\begin{proof} Once we choose a sufficiently small $\va$ in Proposition \ref{prop:optimal choice}, it follows immediately from (3.14) of  Rey \cite{Rey}.
\end{proof}

\begin{lem}\label{lem:w-fact} We have 
 \[
 \int_{\om} u^\frac{n+2}{n-2} w\,\ud x=O(\|w\|^2), \quad \int_{\om} X_0^\frac{n+2}{n-2} w\,\ud x=O(\|w\|^2).
 \]
\end{lem}
\begin{proof}
Since $u=w+\al X_0$ and $w\in E^s$, we have
\[
\int_{\om} u^\frac{n+2}{n-2} w\,\ud x=\al \int_{\om} u^{\frac{4}{n-2}}   X_0 w \,\ud x+ \int_{\om} u^{\frac{4}{n-2}} w^2\,\ud x=  \int_{\om} u^{\frac{4}{n-2}} w^2\,\ud x.
\]
Then the first estimate follows from Proposition \ref{lem:positivity}. Making use of the inequality
\begin{align*}
|a^\frac{n+2}{n-2}-b^\frac{n+2}{n-2}|\le C(a^\frac{4}{n-2}+b^\frac{4}{n-2})|a-b|\quad\forall~a,b>0,
\end{align*}
we have
\[
\left|\int_{\om} (   \al  X_0)^\frac{n+2}{n-2} w\,\ud x- \int_{\om} u^\frac{n+2}{n-2} w\,\ud x\right|\le C\left(\int_{\om} u^{\frac{4}{n-2}} w^2\,\ud x + \int_{\om}  (\al  X_0)^{\frac{4}{n-2}} w^2\,\ud x \right).
\] 
Then the second one follows from Proposition \ref{lem:positivity} as well. The lemma is proved.
\end{proof}

\begin{lem} \label{lem:err-energy} 
We have 
 \[
  \|w\|^2  =O( M_2+ \lda^{1-n}).
  \]
 \end{lem}

 \begin{proof}   
 Making use of the inequality
\begin{align*}
\left|a^\frac{n+2}{n-2}-b^\frac{n+2}{n-2}-\frac{n+2}{n-2}b^{\frac{4}{n-2}}(a-b)\right|\le Cb^{\frac{4}{n-2}-\delta(n)}|a-b|^{1+\delta(n)}+C|a-b|^\frac{n+2}{n-2}
\end{align*}
for all $a,b>0$, where $\delta(n)= \min\{1, \frac{4}{n-2}\}$,  we have
\begingroup
\allowdisplaybreaks
 \begin{align*}
 &u^{\frac{n+2}{n-2}}-\alpha \overline X_0^{\frac{n+2}{n-2}}\\
 &= (\al X_0+w)^{\frac{n+2}{n-2}} - (\alpha X_0)^{\frac{n+2}{n-2}}+ (\alpha  X_0)^{\frac{n+2}{n-2}}- \alpha  X_0^{\frac{n+2}{n-2}} + \alpha  X_0^{\frac{n+2}{n-2}} - \alpha \overline X_0^{\frac{n+2}{n-2}} \\ &
  =   \frac{n+2}{n-2} (\al X_0)^{\frac{4}{n-2}} w +O(X_0^{\frac{4}{n-2}-\delta(n)}|w|^{1+\delta(n)} +|w|^{\frac{n+2}{n-2}}) +(\al^\frac{n+2}{n-2}-\alpha)  X_0^\frac{n+2}{n-2} \\
  &\quad+O(\overline X_0 ^{\frac{4}{n-2}}\lda^{\frac{2-n}{2}}),
 \end{align*}
 \endgroup
and thus, 
\begingroup
\allowdisplaybreaks
 \begin{align}
 -\Delta w&=-\Delta (u-\al X_0)\nonumber \\
 &=  \mathcal{R} u^\frac{n+2}{n-2}- \alpha  r_\infty \overline X_0 ^\frac{n+2}{n-2}\nonumber\\&
 =  (\mathcal{R}-r_\infty) u^\frac{n+2}{n-2} + r_\infty( u^\frac{n+2}{n-2}- \alpha   \overline X_0 ^\frac{n+2}{n-2}) \nonumber \\&
 = (\mathcal{R}-r) u^\frac{n+2}{n-2} + (r-r_\infty) u^\frac{n+2}{n-2} + r_\infty  \frac{n+2}{n-2} (\al X_0)^{\frac{4}{n-2}} w + r_\infty ( \al^\frac{n+2}{n-2}-\al)  X_0^\frac{n+2}{n-2} \nonumber \\& \quad +O(X_0^{\frac{4}{n-2}-\delta(n)}|w|^{1+\delta(n)} +|w|^{\frac{n+2}{n-2}}) +O(\overline X_0 ^{\frac{4}{n-2}}\lda^{\frac{2-n}{2}} ). \label{eq:rltn-1}
 \end{align}
 \endgroup
 Multiplying the above identity by $w$ and  integrating by parts, we obtain 
 \begingroup
\allowdisplaybreaks
 \begin{align*}
 c \|w\|^2 & \le \int_\Omega |\nabla w|^2 \,\ud x-  \frac{n+2}{n-2} r_\infty \int_\Omega  (\al X_0)^{\frac{4}{n-2}} w^2\,\ud x\\ &
  \le  \int_\Omega  |\mathcal{R}-r| |w| u^\frac{n+2}{n-2}   \,\ud x +|r-r_\infty| \left| \int_\Omega    u^\frac{n+2}{n-2} w  \,\ud x\right|+  r_\infty\al |\al^{\frac{4}{n-2}}-1| \left| \int_\Omega X_0^\frac{n+2}{n-2} w\,\ud x\right| \\
  &\quad+C\|w\|^{2+\delta(n)} +C \int_\Omega  \overline X_0 ^{\frac{4}{n-2}}\lda^{\frac{2-n}{2}}  |w|\,\ud x\\ &
 \le C M_2^{1/2} \|w\|+C |r-r_\infty| \|w\|^2  + C\al |\al^{\frac{4}{n-2}}-1| \|w\|^2 \\
 &\quad+C \|w\|^{2+\delta(n)} + \frac{c}{8} \|w\|^2+ C\lda^{1-n} \\ &
 \le \frac{c}{2}\|w\| ^2 + C(\lda^{1-n}+M_2) ,
 \end{align*}
 \endgroup
 where in the first inequality we  used Proposition \ref{lem:positivity}, in the second inequality we used \eqref{eq:rltn-1},  in the third  inequality we used Lemma \ref{lem:w-fact} and
 \begin{align}\label{eq:error25}
 \lda^{\frac{2-n}{2}}\int_{\om }  \overline X_0^{\frac{4}{n-2}} |w| \,\ud x &\le  16 \lda^{\frac{2-n}{2}}\int_{\om }   X_0^{\frac{4}{n-2}} |w| \,\ud x + 16 \lda^{-\frac{2+n}{2}}\int_{\om } |w| \,\ud x \nonumber \\
 & \le C  \lda^{2-n} \int_{\om } \overline X_0^{\frac{4}{n-2}} \,\ud x  +\frac{c}{8}\|w\|^2 + C \lda^{-2-n}\nonumber\\
& \le C \lda^{1-n} + \frac{c}{8}\|w\|^2, 
 \end{align}
 and in the fourth inequality we used
 \begin{align*}
& M_2^{1/2} \|w\| \le \frac{c}{8} \|w\|^2+ \frac{2}{c} M_2 \quad \mbox{by the Cauchy-Schwarz inequality}, \\
& C |r-r_\infty| + C\al |\al^{\frac{4}{n-2}}-1|+\|w\|^{\delta (n)}\le \frac{c}{4} \quad \mbox{for  large }t. 
 \end{align*}
By cancelling the first term on the right-hand side, the lemma follows.
 \end{proof}
 
 \begin{rem} The above estimate is not sharp if $n\ge 4$, since  
 \[
  \lda^{2-n} \int_{\om } \overline X_0^{\frac{4}{n-2}} \le C \lda^{-n} |\ln \lda|^{(5-n)_+}, \quad \mbox{if }n\ge 4.  \]
 
 \end{rem}

 \begin{lem} \label{lem:alpha-r}
 There holds
 \[
 |1-\alpha|=O(M_2^{\frac{1+\delta(n)}{2}}+\lda^{-\frac{n+2}2}+\lda^{2-n}), \quad \mbox{and}\quad  r-r_\infty= O(M_2^{\frac{1+\delta(n)}{2}}+\lda^{-\frac{n+2}2}+\lda^{2-n}),
 \]
  where $\delta(n)= \min\{1, \frac{4}{n-2}\}$.
 \end{lem}
\begin{proof}  Multiplying $u$ to both sides of
\[
-\Delta w= \mathcal{R} u^\frac{n+2}{n-2}-\al r_\infty \overline X_0^\frac{n+2}{n-2}
\]
and integrating over $\om$, we obtain
\[
\int_\Omega w \mathcal{R} u^{\frac{n+2}{n-2}}=r-\al^{-\frac{4}{n-2}} r_\infty \int_\Omega (\alpha \overline X_0)^\frac{n+2}{n-2} u,
\]
where we used  $\int_\Omega u^{\frac{2n}{n-2}}=1$. For the left-hand side, by using Lemma \ref{lem:w-fact} and Lemma \ref{lem:err-energy}, and also the Cauchy-Schwarz inequality, 
\begin{align*}
\left|\int_\Omega w \mathcal{R} u^\frac{n+2}{n-2}\right| &=\left|\int_\Omega (\mathcal{R}-r)  u^\frac{n}{n-2} w u^\frac{2}{n-2}+ r \int_\Omega u^{\frac{n+2}{n-2}}w\right|\\
&\le \left|\int_\Omega (\mathcal{R}-r)^2  u^\frac{2n}{n-2} \right| + \left|\int_\Omega w^2 u^\frac{4}{n-2}\right| +  \left| r \int_\Omega u^{\frac{n+2}{n-2}}w\right|\\
& =O(M_2+\lda^{1-n}).
\end{align*}
For the last term on the right-hand side,  
 \begingroup
\allowdisplaybreaks
\begin{align*}
\int_\Omega (\alpha \overline X_0)^\frac{n+2}{n-2} u-\int_\Omega u^{\frac{2n}{n-2}}&=\int_\Omega [(\alpha \overline X_0)^\frac{n+2}{n-2}-(\alpha X_0)^\frac{n+2}{n-2}] u
+ \int_\Omega [(\alpha X_0)^\frac{n+2}{n-2}-u^\frac{n+2}{n-2} ] u \\& =O(\lda^{-\frac{n+2}2}+\lda^{2-n}) - \frac{n+2}{n-2} \int_\Omega u^{\frac{4}{n-2}} w u+O(\|w\|^{1+\delta(n)})\\&
=O(\lda^{-\frac{n+2}2}+\lda^{2-n})+O(\|w\|^{1+\delta(n)})\\&
= O(M_2^{\frac{1+\delta(n)}{2}}+\lda^{-\frac{n+2}2}+\lda^{2-n}),
\end{align*}
\endgroup
where we used 
\begin{align*}
\int_\Omega [(\alpha \overline X_0)^\frac{n+2}{n-2}-(\alpha X_0)^\frac{n+2}{n-2}] u & \le C \lda^{\frac{2-n}{2}} \int_\Omega \overline X_0^{\frac{4}{n-2}}u \\& \le C \lda^{\frac{2-n}{2}} \Big( \int_\Omega \overline X_0^{\frac{8n}{(n-2)(n+2)}}\Big)^{\frac{n+2}{2n}}\\&
 \le C  (\lda^{-\frac{n+2}2}+\lda^{2-n}). 
\end{align*}
Since  $\int_\Omega u^{\frac{2n}{n-2}}=1$, we obtain 
\be \label{eq:alpha-r-1}
r-\al^{-\frac{4}{n-2}} r_\infty = O(M_2^{\frac{1+\delta(n)}{2}}+\lda^{-\frac{n+2}2}+\lda^{2-n}).
\ee

On the other hand, since
\begin{align*}
\int_{\R^n} |\nabla \overline X_0|^2 - \int_{\om} |\nabla X_0|^2
& = \int_{\om} (-\Delta \overline X_0)\overline X_0 +\int_{\R^n\setminus\om} (-\Delta \overline X_0)\overline X_0 - \int_{\om} (-\Delta  X_0) X_0 \\
& = r_\infty \int_{\om} \overline X_0^\frac{n+2}{n-2} (\overline X_0 - X_0) +r_\infty \int_{\R^n\setminus\om} \overline X_0^\frac{2n}{n-2}\\
& = O(\lambda^{2-n})
\end{align*}
and
\begin{align*}
\int_{\om} \nabla X_0 \nabla w&= r_\infty \int_{\om} \overline X_0^{\frac{n+2}{n-2}} w\\
&=  r_\infty  \int_{\om }  X_0^{\frac{n+2}{n-2}} w + O\left(  \lda^{\frac{2-n}{2}}\int_{\om }  \overline X_0^{\frac{4}{n-2}} |w| \,\ud x\right)\\
&=  r_\infty  \int_{\om }  X_0^{\frac{n+2}{n-2}} w + O(\lda^{1-n}+\|w\|^2)\quad (\mbox{by }\eqref{eq:error25})\\
&= O(\lda^{1-n}+\|w\|^2) \quad \mbox{(by Lemma \ref{lem:w-fact})},
\end{align*}
we have
 \begingroup
\allowdisplaybreaks
\begin{align*}
0\le r-r_\infty&= \int_{\om} |\nabla u|^2 -\int_{\R^n} |\nabla \overline X_0|^2 \\&
= \int_{\om} (\al^2 |\nabla X_0|^2 +2\al \nabla X_0 \nabla w +|\nabla w|^2  )- \int_{\om} |\nabla X_0|^2 +O(\lda^{2-n})\\&
 = (\al^2-1)\int_{\om} |\nabla X_0|^2 +O(\lda^{2-n}+M_2)\\&
 = (\al^2-1)\int_{\R^n} |\nabla\overline  X_0|^2+O(\lda^{2-n}+M_2) \\&
 =(\al^2-1)r_\infty+O(\lda^{2-n}+M_2),
\end{align*}
\endgroup
where we used Lemma \ref{lem:err-energy} in the third equality. 
Hence, 
\be \label{eq:alpha-r-2}
r-\al^2 r_\infty = O(\lda^{2-n}+M_2). 
\ee
Subtracting \eqref{eq:alpha-r-2} from \eqref{eq:alpha-r-1}, we obtain 
\[
\al^2 - \al^{-\frac{4}{n-2}} =  O(M_2^{\frac{1+\delta(n)}{2}}+\lda^{-\frac{n+2}2}+\lda^{2-n}).
\]
Since $|\alpha -1|$ is very small, 
\[
\al-1= O(M_2^{\frac{1+\delta(n)}{2}}+\lda^{-\frac{n+2}2}+\lda^{2-n}).
\]
Substituting it into \eqref{eq:alpha-r-2}, we obtain  
\[
r-r_\infty= (\al^2-1)r_\infty+O(\lda^{2-n}+M_2) = O(M_2^{\frac{1+\delta(n)}{2}}+\lda^{-\frac{n+2}2}+\lda^{2-n}).
\]
The lemma is proved.
\end{proof}
 
A related estimate of $|\al-1|$ in the whole space has been proved in Ciraolo-Figalli-Maggi \cite{CFM}.

\begin{prop} \label{lem:relative-1}
We have
\[
\left\|\frac{u}{X_0}-1\right\|_{L^{\infty}(\om)}\le C(M_2^{\frac{2}{n+2}}+\lda^{-\frac{2}{n+2}}).
\]
\end{prop}
\begin{proof}
Let 
\begin{align*}
u_\lda(y,t)&= \lda^{-\frac{n-2}{2}} u(a+ \lda^{-1} y,t),  &  Z &= \lda^{-\frac{n-2}{2}} X_0(a+ \lda^{-1} y),\\ 
\mathcal{R}_\lda(y,t)&= \mathcal{R}(a+ \lda^{-1} y,t),  & \overline Z&= \lda^{-\frac{n-2}{2}} \overline X_0(a+ \lda^{-1} y)=r_\infty^{-\frac{n-2}{4}}U_{0,1}
\end{align*}
with $y\in \om(t):= \{y\in \R^n:a+ \lda^{-1} y\in \om\}$. Then 
\begin{align}
&-\Delta u_\lda =\mathcal{R}_\lda u_{\lda}^{\frac{n+2}{n-2}}\ \  \mbox{in }\om(t), \quad -\Delta Z= r_\infty \overline Z^{\frac{n+2}{n-2}} \ \  \mbox{in }\om(t), \quad u_\lda=Z=0 \quad \mbox{on }\pa \om(t),\nonumber\\
&\lim_{t\to \infty} \int_{\om(t)} |u_\lda-Z|^{\frac{2n}{n-2}}\,\ud y= \lim_{t\to \infty}  \int_{\om} |u-X_0|^{\frac{2n}{n-2}}\,\ud x=0, \label{eq:relativerrorLnorm}
\end{align}
and $\| \mathcal{R}_\lda-r_\infty \|_{L^{\infty}(\om(t))}\to 0$ as $t\to \infty$. For any $R>2$ with  $B_{2R}\subset \om(t)$, it is clear that $\frac{1}{C(R)}\le Z\le C$  in $B_{R}$, where $C(R)>0$ depends  only on $n$ and $R$. By applying the  Moser iterations, one also obtains
\[
\frac{1}{C(R)}\le u_\lda\le C \quad \mbox{in }B_{R}
\]
with possibly different constants. Define the Kelvin transform:
\[
\widetilde u_\lda(y,t)=\frac{1}{|y|^{n-2}}u_\lda\left(\frac{y}{|y|^2},t\right),\quad  \widetilde Z(y,t)=\frac{1}{|y|^{n-2}}Z\left(\frac{y}{|y|^2},t\right).
\]
Then
\[
-\Delta \widetilde u_\lda =\mathcal{R}_\lda\left(\frac{y}{|y|^2},t\right) \widetilde u_{\lda}^{\frac{n+2}{n-2}}\quad\mbox{and}\quad -\Delta \widetilde Z= r_\infty \overline Z^{\frac{n+2}{n-2}}
\]
in $\widetilde \Omega(t):= \{y: \frac{y}{|y|^2}\in\Omega(t)\}$.

Since 
\[
-\Delta(u_\lda-Z)= (\mathcal{R}_\lda-r_\infty) \overline Z^{\frac{n+2}{n-2}} + \mathcal{R}_\lda (u_\lda^{\frac{n+2}{n-2}}-  Z^{\frac{n+2}{n-2}})+\mathcal{R}_\lda (Z^{\frac{n+2}{n-2}}- \overline Z^{\frac{n+2}{n-2}}),
\]
by applying the Moser iteration to this equation with the help of \eqref{eq:relativerrorLnorm}, we obtain
\be \label{eq:appendix-b-1}
\left\|u_\lda- Z\right\|_{L^\infty(B_{1})}\le C \varepsilon(t),
\ee
where 
\begin{align*}
\varepsilon(t)&=\left(\int_{\Omega(t)}|\mathcal{R}_\lda-r_\infty|^\frac{n+2}{2}\overline Z^{\frac{2n}{n-2}}\,\ud x\right)^{\frac{2}{n+2}}  +\|u_\lda-Z\|_{L^{\frac{2n}{n-2}}(\Omega(t))} + \|Z-\overline Z\|_{L^\infty(\Omega(t))}\\
&=\left(\int_{\Omega}|\mathcal{R}-r_\infty|^\frac{n+2}{2}\overline X_0^{\frac{2n}{n-2}}\,\ud x\right)^{\frac{2}{n+2}}  +\|u-X_0\|_{L^{\frac{2n}{n-2}}(\Omega(t))} + \lda^{-\frac{n-2}{2}}\|X_0-\overline X_0\|_{L^\infty(\Omega(t))}
\end{align*}
and $C$ depends only on $n$. Applying the same argument to the equation of $\widetilde u_\lda-\widetilde Z$, with noticing that $u_\lda=Z=0$ on $\partial\Omega(t)$ and with the help of \eqref{eq:relativerrorLnorm}, we have
\[
\left\|\widetilde u_\lda-\widetilde Z\right\|_{L^\infty(\widetilde \Omega(t)\cap B_{1})}\le C \varepsilon(t),
\]
which implies that
\begin{equation}\label{eq:appendix-b-3}
\begin{split}
\left\||y|^{n-2}(u_\lda-Z)(y,t)\right\|_{L^\infty(\om(t)\setminus B_{1})}\le C \varepsilon(t).
\end{split}
\end{equation}

Scaling the estimates \eqref{eq:appendix-b-1} and \eqref{eq:appendix-b-3} back to $u$, we have
\begin{align}
\left\| \frac{u}{X_0} -1 \right\|_{L^\infty(B_{1/\lda}(a))} & \le C \varepsilon(t),\label{eq:appendix-b-4}\\
|u(x,t)-X_0(x,t)| &\le C \varepsilon(t)\lda^{\frac{2-n}{2}}|x-a|^{2-n},\quad\forall~x\in \Omega\setminus B_{1/\lda}(a). \nonumber
\end{align}
If $x\in \Omega\setminus B_{1/\lda}(a)$, then $\lda^{\frac{2-n}{2}}|x-a|^{2-n}\le C \overline X_0$, and thus
\begin{align*}
|u(x,t)-X_0(x,t)| &\le C \varepsilon(t) \overline X_0\le C \varepsilon(t) (X_0+\lda^{\frac{2-n}{2}}).
\end{align*}
Let $\delta_0$ be the one in Proposition \ref{prop:jx20}. Let $\delta_1>0$ be the uniform radius of interior balls tangent to each point on $\partial\Omega$. Let $\delta_2=\min(\delta_1, \delta_0/8)$. Then $X_0\ge C\lda^{\frac{2-n}{a}}$ in $\Omega_{\delta_2/8}:=\{x\in\Omega: \mbox{dist}(x,\partial\Omega)>\delta_2/8\}$. Hence,
\[
|u(x,t)-X_0(x,t)| \le C \varepsilon(t) X_0(x,t),\quad\forall~x\in \Omega_{\delta_2/8}\setminus B_{1/\lda}(a).
\]
Combining \eqref{eq:appendix-b-4}, we obtain
\begin{align}
\left\| \frac{u}{X_0} -1 \right\|_{L^\infty(\Omega_{\delta_2/8})} & \le C \varepsilon(t). \label{eq:appendix-b-6}
\end{align}

On one hand, by the Hopf lemma and elliptic estimates for the equation of $X_0$, one has
\[
X_0\ge C^{-1} \lda^{\frac{2-n}{2}}\mbox{dist}(x,\partial\Omega)\quad\mbox{in }\Omega\setminus \Omega_{\delta_2/8}.
\]
On the other hand, by applying the $W^{2,2n+2}$ estimate (Theorem 9.13 in Gilbarg-Trudinger \cite{GT2001}) to
\begin{align*}
-\Delta(u-X_0)&=(\mathcal{R}-r_\infty) \overline X_0^{\frac{n+2}{n-2}} + \mathcal{R} (u^{\frac{n+2}{n-2}}-  X_0^{\frac{n+2}{n-2}})+\mathcal{R} (X_0^{\frac{n+2}{n-2}}- \overline X_0^{\frac{n+2}{n-2}})\quad\mbox{in }\Omega\setminus \Omega_{\delta_2/2},\\
u-X_0&=0\quad \mbox{on }\partial\Omega,
\end{align*}
and using the Sobolev embeddings, we obtain
\[
\|u-X_0\|_{C^1(\Omega\setminus \Omega_{\delta_2/4})}\le C \|u-X_0\|_{W^{2,2n+2}(\Omega\setminus \Omega_{\delta_2/2})}\le C\varepsilon(t)\lda^{\frac{2-n}{2}}+C\lda^{-\frac{2+n}{2}}.
\]
Hence, 
\begin{align*}
\left\| \frac{u}{X_0} -1 \right\|_{L^\infty(\Omega\setminus \Omega_{\delta_2/8})} & \le C \varepsilon(t)+C\lda^{-2}.
\end{align*}
Together with \eqref{eq:appendix-b-6}, it follows that
\begin{align*}
\lim_{t\to\infty}\left\| \frac{u}{X_0} -1 \right\|_{L^\infty(\Omega)} =0.
\end{align*}
Consequently, $\overline X_0\le C(u + \lda^{\frac{2-n}{2}})$, and thus,
\begin{align*}
\varepsilon(t)\le CM_2^{\frac{2}{n+2}}+C\lda^{-\frac{2}{n+2}}+C|r-r_\infty|+\|w\|+C |1-\alpha| + C\lda^{2-n}\le CM_2^{\frac{2}{n+2}}+C\lda^{-\frac{2}{n+2}}.
\end{align*}
Therefore,
\[
\left\| \frac{u}{X_0} -1 \right\|_{L^\infty(\Omega)} \le C \varepsilon(t)+C\lda^{-2}\le CM_2^{\frac{2}{n+2}}+C\lda^{-\frac{2}{n+2}}.
\]
\end{proof}

\begin{lem} \label{lem:kernelpt}
We have for every $j=1,\dots,n+1$ that
\begin{align*}
|X_j|&\le C(u+\lda^{\frac{2-n}{2}}),\\
|\partial_t X_j|&\le C \left|\left(\lda \dot a , \frac{\dot \lda }{\lda } \right)\right|(u+\lda^{\frac{2-n}{2}}).
\end{align*}
\end{lem}
\begin{proof}
By Proposition \ref{prop:bubble-derivative} and Proposition \ref{lem:relative-1}, we have for $i=1,\dots,n$ that
\begin{align*}
|X_i|&= r_\infty^{-\frac{n-2}{4}}\lda^{-1}|\pa_{a^i} U_{a,\lda}-\pa_{a^i} h_{a,\lda}|\le C (U_{a,\lda}+\lda^{-\frac{n}{2}})\le C(X_0+\lda^{\frac{2-n}{2}})\le C(u+\lda^{\frac{2-n}{2}}),
\end{align*}
and
\begin{align*}
|X_{n+1}|&= r_\infty^{-\frac{n-2}{4}}\lda |\pa_{\lda} U_{a,\lda}-\pa_{\lda} h_{a,\lda}|\le C (U_{a,\lda}+\lda^{-\frac{n-2}{2}})\le C(X_0+\lda^{\frac{2-n}{2}})\le C(u+\lda^{\frac{2-n}{2}}).
\end{align*}
Similar calculations show that
\[
\frac{1}{\lda}|\nabla_a X_j|+\lda|\nabla_\lda X_j|\le C(u+\lda^{\frac{2-n}{2}})
\]
for all $j=1,\dots,n+1$. Therefore,
\[
|\partial_t X_j|= |\nabla_a X_j \cdot \dot{a} + \partial_\lda X_j\cdot  \dot{\lda} | \le C \left|\left(\lda \dot a , \frac{\dot \lda }{\lda } \right)\right|(u+\lda^{\frac{2-n}{2}}).
\]
\end{proof}

\begin{lem}\label{lem:gradientnormest}
We have
\begin{equation}\label{eq:auxH1norm}
\begin{split}
\int_\Omega |\nabla X_0|^2&= (1+o(1))r_\infty  \|X_0\|^2_{\Lt},\\
\int_\Omega |\nabla X_j|^2&= (1+o(1))r_\infty  \frac{n+2}{n-2} \|X_j\|^2_{\Lt}
\end{split}
\end{equation}
for every $j=1,\dots, n+1$, and
\begin{align}
\int_\Omega \nabla X_j\cdot \nabla X_k= o(1)(\|X_j\|^2_{\Lt}+\|X_k\|^2_{\Lt}) \label{eq:auxH1acorss}
\end{align}
for every $k=0, 1,\dots, n+1, k\neq j$. 
\end{lem}

\begin{proof}
For $j\ge 1$, we have
\begin{align*}
\int_\Omega |\nabla X_j|^2= -\int_\Omega \Delta X_j \cdot X_j=r_\infty  \frac{n+2}{n-2}  \int_\Omega \overline X_0 ^{\frac{4}{n-2}}  (X_j +O(\lda^{\frac{2-n}2 })) X_j.
\end{align*}
Since
 \begingroup
\allowdisplaybreaks
\begin{align*}
 &\int_\Omega (\alpha\overline X_0)^{\frac{4}{n-2}} X_j^2\\
 &=\int_\Omega [(\alpha\overline X_0)^{\frac{4}{n-2}}-(\alpha X_0)^{\frac{4}{n-2}}] X_j^2+ \int_\Omega [(\alpha X_0)^{\frac{4}{n-2}}-u^{\frac{4}{n-2}}] X_j^2 + \int_\Omega u^{\frac{4}{n-2}} X_j^2\\
 &= o(1) \|X_j\|^2_{L^2}+ (1+o(1))\|X_j\|^2_{\Lt}\\
  &=o(1) \|X_j\|^2_{H_0^1}+ (1+o(1))\|X_j\|^2_{\Lt}
\end{align*}
\endgroup
and
\begin{align}
\lda^{\frac{2-n}2 }\int_\Omega \overline X_0^{\frac{4}{n-2}} |X_j|&\le \lda^{\frac{2-n}2 }\int_\Omega \overline X_0^{\frac{4}{n-2}}+\lda^{\frac{2-n}2 }\int_\Omega \overline X_0^{\frac{4}{n-2}} |X_j|^2\nonumber\\
&=o(1)+o(1) \|X_j\|^2_{H_0^1}+ o(1)\|X_j\|^2_{\Lt},\label{eq:auxcross}
\end{align}
by using \eqref{eq:almost-orth-basis} and Lemma \ref{lem:alpha-r}, we obtained \eqref{eq:auxH1norm} for $j\ge 1$. The estimate for $j=0$ is similar.

For $k\ge 1, k\neq j$, we have
\begin{align*}
\int_\Omega \nabla X_j\cdot \nabla X_k= -\int_\Omega \Delta X_j \cdot X_k=r_\infty  \frac{n+2}{n-2}  \int_\Omega \overline X_0 ^{\frac{4}{n-2}}  (X_j +O(\lda^{\frac{2-n}2 })) X_k.
\end{align*}
Since
 \begingroup
\allowdisplaybreaks
\begin{align*}
 &\int_\Omega (\alpha\overline X_0)^{\frac{4}{n-2}} X_jX_k\\
 &=\int_\Omega [(\alpha\overline X_0)^{\frac{4}{n-2}}-(\alpha X_0)^{\frac{4}{n-2}}] X_jX_k+ \int_\Omega [(\alpha X_0)^{\frac{4}{n-2}}-u^{\frac{4}{n-2}}] X_jX_k + \int_\Omega u^{\frac{4}{n-2}} X_jX_k\\
 &= o(1) (\|X_j\|^2_{L^2}+\|X_k\|^2_{L^2})+ o(1)(\|X_j\|^2_{\Lt}+\|X_j\|^2_{\Lt})+o(1)\\
  &=o(1) (\|X_j\|^2_{H_0^1}+\|X_k\|^2_{H_0^1})+ o(1)(\|X_j\|^2_{\Lt}+\|X_k\|^2_{\Lt}),
\end{align*}
\endgroup
together with \eqref{eq:auxcross} and \eqref{eq:auxH1norm}, we obtain \eqref{eq:auxH1acorss} for $k\ge 1$. The estimate for $k=0$ is similar.
\end{proof}

The following quantity will play as the leading term in the end.
\begin{prop}
There holds
\begin{align}
\int_{\pa \om}  |\nabla PU_{a,\lda}|^2 \langle x-a, \nu\rangle \,\ud S &= C_2(n)H(a,a)\lda^{2-n}+O(\lda^{1-n}),\label{eq:phoest}
\end{align}
where
\[
C_2(n)= \frac{(n-2)(n+2)}{n}[ n(n-2)]^{\frac{n+ 2}{4}} \int_{\R^n} (1+|x|^2)^{-\frac{n+4}{2}} |x|^2\,\ud x.  
\]
\end{prop}

\begin{proof}
Multiplying both sides of
\[
-\Delta PU_{a,\lda} = U_{a,\lda}^{\frac{n+2}{n-2}} \quad \mbox{in }\om, \quad PU_{a,\lda}=0 \quad \mbox{on }\pa \om,
\]
by $(x-a) \cdot \nabla PU_{a,\lda}$, we obtain
\begingroup
\allowdisplaybreaks
\begin{align*}
&\int_{\om} ((x-a) \cdot \nabla PU_{a,\lda}) U_{a,\lda}^{\frac{n+2}{n-2}}\,\ud x \\
&= -\int_{\om} ((x-a) \cdot \nabla PU_{a,\lda})  \Delta PU_{a,\lda} \,\ud x\\&
= -\frac{n-2}{2} \int_{\om }  |\nabla PU_{a,\lda}|^2\,\ud x -\frac{1}{2} \int_{\pa \om} |\nabla PU_{a,\lda}|^2 \langle x-a, \nu\rangle \,\ud S\\&
= - \frac{n-2}{2}\int_\om PU_{a,\lda} \cdot  U_{a,\lda}^{\frac{n+2}{n-2}}\,\ud x-\frac{1}{2} \int_{\pa \om} |\nabla PU_{a,\lda}|^2 \langle x-a, \nu\rangle\,\ud S.
\end{align*}
\endgroup
On the other hand,
\begin{align*}
&\int_\om ((x-a) \cdot \nabla PU_{a,\lda})  U_{a,\lda}^{\frac{n+2}{n-2}}\,\ud x\\
&=-n \int_\om PU_{a,\lda} \cdot U_{a,\lda}^{\frac{n+2}{n-2}}\,\ud x -\frac{n+2}{n-2} \int_\om PU_{a,\lda} ((x-a) \cdot \nabla U_{a,\lda}) U_{a,\lda}^{\frac{4}{n-2}}\,\ud x
\end{align*}
and
\begin{align*}
&PU_{a,\lda} ((x-a)\cdot \nabla U_{a,\lda}) \\
&=(U_{a,\lda}- h_{a,\lda}) [(x-a)\cdot\nabla (PU_{a,\lda} +h_{a,\lda})]\\
&= ((x-a) \cdot \nabla PU_{a,\lda}) U_{a,\lda} +((x-a)\cdot\nabla h_{a,\lda}) U_{a,\lda}-((x-a)\cdot\nabla U_{a,\lda})h_{a,\lda}.
\end{align*}
Hence,
\begin{align*}
&\int_\om ((x-a) \cdot \nabla PU_{a,\lda})  U_{a,\lda}^{\frac{n+2}{n-2}}  \,\ud x\\
&=-\frac{n-2}{2}\int_\om PU_{a,\lda} \cdot U_{a,\lda}^{\frac{n+2}{n-2}}  \,\ud x\\
&\quad - \frac{n+2}{2n} \int _\om \left[((x-a)\nabla h_{a,\lda}) U_{a,\lda}^{\frac{n+2}{n-2}}-  ((x-a)\nabla U_{a,\lda}) h_{a,\lda} U_{a,\lda}^{\frac{4}{n-2}}\right]\,\ud x.
\end{align*}
Since
\begin{align*}
\int_\om ((x-a)\nabla h_{a,\lda}) U_{a,\lda}^{\frac{n+2}{n-2}}\,\ud x&=\nabla h_{a,\lda}(a)\int_\om (x-a)  U_{a,\lda}^{\frac{n+2}{n-2}} \,\ud x+ O( \lda^{\frac{2-n}{2}}\int_\om |x-a|^2 U_{a,\lda}^{\frac{n+2}{n-2}}\,\ud x)\\
&=O(\lda^{-n}\ln \lda),
\end{align*}
and by Proposition \ref{prop:bubble-derivative},
\begin{align*}
\int_{\om}  ((x-a)\cdot \nabla U_{a,\lda}) h_{a,\lda} U_{a,\lda}^{\frac{4}{n-2}} \,\ud x&= (2-n)h_{a,\lda}(a) \int_{\om} U_{a,\lda}^{\frac{n+2}{n-2}} \frac{\lda^2 |x|^2}{1+\lda^2|x|^2}\,\ud x  + O(\lda^{1-n})\\&
=-C(n) H(a,a)\lda^{ 2-n } +  O(\lda^{1-n}),
\end{align*}
where 
\[
C(n)=(n-2) [ n(n-2)]^{\frac{n+ 2}{4}} \int_{\R^n} (1+|x|^2)^{-\frac{n+4}{2}} |x|^2\,\ud x, 
\]
we have
\begin{align*}
&\int_{\om} ((x-a) \cdot\nabla PU_{a,\lda})  U_{a,\lda}^{\frac{n+2}{n-2}}  \,\ud x\\
&=-\frac{n-2}{2}\int_{\om} PU_{a,\lda}\cdot U_{a,\lda}^{\frac{n+2}{n-2}} \,\ud x-\frac{n+2}{2n} C(n) H(a,a) \lda^{2-n} +  O(\lda^{1-n}).
\end{align*}
Therefore,
\begin{align*}
\int_{\pa \om} |\nabla PU_{a,\lda}|^2 \langle x-a, \nu\rangle\,\ud S= \frac{n+2}{n} C(n) H(a,a) \lda^{2-n} +  O(\lda^{1-n}) .
\end{align*}
This proves \eqref{eq:phoest}.
\end{proof}

\begin{lem} \label{lem:H2}
We have
\[
\|w\|_{W^{2,2}_0(\om)} \le C (\lda^{\frac{1-n}2}+M_2^{1/2}) \lda.
\]
\end{lem}

\begin{proof} Since 
\begin{align*}
 -\Delta w=\mathcal{R} u^\frac{n+2}{n-2}- \alpha  r_\infty \overline X_0 ^\frac{n+2}{n-2}= (\mathcal{R}-r) u^\frac{n+2}{n-2} +(r-r_\infty) u^\frac{n+2}{n-2} +r_\infty (u^\frac{n+2}{n-2}-\al \overline X_0 ^\frac{n+2}{n-2}),
\end{align*}
it suffices to estimate the $L^2$ norm of the right hand side, as the lemma then follows from the $W^{2,2}$ estimates for the Poisson equation. 

By using Proposition \ref{lem:relative-1},  we obtain 
\[
\|(\mathcal{R}-r) u^\frac{n+2}{n-2} \|_{L^2}\le C\lda \|(\mathcal{R}-r) u^\frac{n}{n-2} \|_{L^2}= C\lda M_2^{1/2}.
\]
By using Proposition \ref{lem:relative-1} and Lemma \ref{lem:alpha-r}, we obtain 
\[
\|(r-r_\infty) u^\frac{n+2}{n-2} \|_{L^2} \le C\lda |r-r_\infty| \le C\lda (M_2^{1/2}+\lda^{-\frac{n+2}2}+\lda^{2-n}).
\]
Since 
\begin{align}
|u^\frac{n+2}{n-2}-\al \overline X_0 ^\frac{n+2}{n-2}| &=| u^\frac{n+2}{n-2}- (\al X_0)^\frac{n+2}{n-2} + (\al X_0)^\frac{n+2}{n-2}- (\al \overline X_0)^\frac{n+2}{n-2}+(\al \overline X_0)^\frac{n+2}{n-2}-\al \overline X_0 ^\frac{n+2}{n-2} |\nonumber \\&
\le C (\al X_0)^\frac{4}{n-2} |w|+   C\lda^{\frac{2-n}{2}} \overline X_0^{\frac{4}{n-2}} +C|\al-1|\overline X_0 ^\frac{n+2}{n-2}, \label{eq:traceerroraux}
\end{align}
by using Lemma \ref{lem:alpha-r} and Lemma \ref{lem:err-energy}, we obtain
 \begingroup
\allowdisplaybreaks
\begin{align*}
\| u^\frac{n+2}{n-2}-\al \overline X_0 ^\frac{n+2}{n-2} \|_{L^2(\om)} \le C\lambda \|w\| + C(\lda^{3-n}+\lda^{-\frac{n}{2}}) + C|\alpha-1| \lambda \le  C (\lda^{\frac{1-n}2}+M_2^{1/2})\lda.
\end{align*}
\endgroup
The lemma is proved. 
\end{proof}

\begin{lem} \label{lem:traceH1}
We have
\[
\|\nabla w\|^2_{L^2(\partial\om)} \le C(M_2+\lda^{1-n}).
\]
\end{lem}
\begin{proof}
Let $\delta_0$ be the one in Proposition \ref{prop:jx20}, and $\Omega_{\delta_0}:=\{x\in\Omega: \mbox{dist}(x,\partial\Omega)>\delta_0\}$. 
We have
\begin{align*}
-\Delta w&=(\mathcal{R}-r) u^\frac{n+2}{n-2} +(r-r_\infty) u^\frac{n+2}{n-2} +r_\infty (u^\frac{n+2}{n-2}-\al \overline X_0 ^\frac{n+2}{n-2})\quad\mbox{in }\Omega,\\
w&=0\quad\mbox{on }\partial\Omega.
\end{align*}
Since $|u|+|X_0|+|\overline X_0|\le C \lda^{\frac{2-n}{2}}$ in $\Omega\setminus\Omega_{\delta_0}$, we have 
\begin{align*}
\|(\mathcal{R}-r) u^\frac{n+2}{n-2} \|^2_{L^2(\Omega\setminus\Omega_{\delta_0})}&\le CM_2,\\
\|(r-r_\infty) u^\frac{n+2}{n-2} \|_{L^\infty(\Omega\setminus\Omega_{\delta_0})}&\le C\|r-r_\infty\|^2_{L^\infty(\Omega)}+ C\|u^\frac{n+2}{n-2}\|^2_{L^\infty(\Omega\setminus\Omega_{\delta_0})}\le C(M_2+\lda^{1-n}),\\
\|r_\infty (u^\frac{n+2}{n-2}-\al \overline X_0 ^\frac{n+2}{n-2})\|^2_{L^2(\Omega\setminus\Omega_{\delta_0})}&\le C(M_2+\lda^{1-n}),
\end{align*}
where we used Lemma \ref{lem:alpha-r} in the second inequality, and Lemma \ref{lem:alpha-r}  and \eqref{eq:traceerroraux} in the last inequality. Therefore, by the $W^{2,2}$ estimate (Theorem 9.13 in Gilbarg-Trudinger \cite{GT2001}) and the trace embedding, we obtain
\[
\|\nabla w\|^2_{L^2(\partial\om)} \le \|w\|^2_{W^{2,2}(\om\setminus\Omega_{\delta_0/2})} \le C(M_2+\lda^{1-n}).
\]
\end{proof}

\subsection{Finite dimensional reduction}

We shall project the $L^2$ gradient flow
\[
\frac{n+2}{n-2}\pa_t u= -(\mathcal{R}-r) u
\]
into $E^{ut}$, $ E^{c}$ and $E^s$,  and single out the leading term in the total $L^2$ quantity $M_2$. Let
\begin{align*}
b_j&= -\langle (\mathcal{R}-r) u, X_j\rangle_{\Lt}, \quad j=0,\dots, n+1,\\
B&=(b_0,\dots, b_{n+1}). 
\end{align*}
Using  H\"older's inequality and \eqref{eq:almost-orth-basis}, we have the general upper bound 
\[
|b_j| \le C M_2^{1/2}.
\]
But  the projection in the unstable direction is actually much smaller. 
\begin{lem}\label{lem:unstable-small} We have 
\[
|b_0|\le C (\lda^{1-n} +M_{2})^{1/2} \cdot  M_2^{1/2}.
\]
\end{lem}

\begin{proof}
Using Lemma \ref{lem:err-energy}, we have
 \begingroup
\allowdisplaybreaks
\begin{align*}
|\alpha b_0|&=\left|\int_\Omega (\mathcal{R}-r)u^\frac{n+2}{n-2} \alpha  X_0 \right| \\&
 =\left|\int_\Omega (\mathcal{R}-r)u^\frac{n+2}{n-2} (u-w) \right|\\&
 =  \left|\int_\Omega (\mathcal{R}-r)u^\frac{n+2}{n-2} w\right| \\&
 \le M_{2}^{1/2}\left( \int_{\om} w^2 u^{\frac{4}{n-2}}\right)^{1/2} \\&
 \le C (\lda^{1-n} +M_{2})^{1/2} \cdot  M_2^{1/2},
\end{align*}
\endgroup
where we used Lemma \ref{lem:err-energy} in the last inequality.  The lemma is proved.
\end{proof}

\begin{rem} 
It is Lemma \ref{lem:unstable-small} in which we used the volume preservation crucially. 
\end{rem}

Next, we derive a finite dimensional flow.

\begin{lem} \label{lem:tangent-flow}
There holds 
 \begin{align*}
 \left( \frac{\dot\alpha}{\alpha},    \lda \dot a , \frac{  \dot \lda }{\lda } \right)&= \left(\frac{n-2}{n+2}+o (1)\right) \left(\frac{b_0}{\kappa_0}, \dots \frac{b_{n+1}}{\kappa_{n+1}}\right) + O(M_2^{1/2}+\lda^{\frac{1-n}{2}})|B|+ O(M_2+\lda^{1-n}),
 \end{align*}
 where $ \kappa_0,\dots \kappa_{n+1}$ are the positive constants in \eqref{eq:almost-orth-basis}.
\end{lem}

\begin{proof} The proof is the same as that of Lemma 4.1 in Mayer \cite{Mayer}.  
First of all,
\begin{align*}
\pa_t u&= \pa_t (\al X_0)+\pa_t w\\&
=  \left( \dot\alpha ,   \al  \lda  \dot a  , \frac{ \al  \dot \lda}{\lda } \right) \cdot X +\pa_t w,
\end{align*}
where $X=(X_0,\dots, X_{n+1})$. 
Differentiating $
0=\langle w, X_{j}\rangle_{\Lt}$ in $t$, and using \eqref{eq:uteq} and Lemma \ref{lem:ppty}, we see that
\begin{align*}
0&=\langle \pa_t w, X_{j}\rangle_{\Lt} + \int_\Omega wX_{j} \pa_t u^{\frac{4}{n-2}}+ \langle w, \pa_t X_{j}\rangle_{\Lt}\\&
=-  \Big\langle  \big( \dot\alpha,   \al \lda \dot a , \frac{ \al  \dot \lda }{\lda } \big)\cdot X+ \frac{n-2}{n+2}  (\mathcal{R}-r) u  +\frac{4}{n+2} (\mathcal{R}-r) w , X_{j}\Big\rangle_{\Lt} + \langle w, \pa_t X_{j}\rangle_{\Lt}\\
&=-  \Big\langle  \big( \dot\alpha,   \al \lda \dot a , \frac{ \al  \dot \lda }{\lda } \big)\cdot X, X_{j}\Big\rangle_{\Lt} + \frac{n-2}{n+2} b_j  -\Big\langle \frac{4}{n+2} (\mathcal{R}-r) w , X_{j}\Big\rangle_{\Lt} + \langle w, \pa_t X_{j}\rangle_{\Lt} .
\end{align*}
By  Lemma \ref{lem:kernelpt}, we have
 \begingroup
\allowdisplaybreaks
\begin{align*}
\int_\Omega  |(\mathcal{R}-r) w X_i| u^{\frac{4}{n-2}}&\le \int_\Omega  |\mathcal{R}-r| |w| | u^{\frac{n+2}{n-2}}+C\lda^{\frac{2-n}{2}}\int_\Omega  |\mathcal{R}-r| |w| u^{\frac{4}{n-2}}\\
&\le M_2^{1/2}\|w\|_{\Lt} + C\lda^{\frac{2-n}{2}}\int_\Omega   |w| \overline X_0^{\frac{4}{n-2}}\\
&\le M_2+C\|w\|^2+C\lda^{1-n}\\
&=O(M_2+\lda^{1-n}),
\end{align*}
\endgroup
where we used the fact that $w\in E^s$, Proposition \ref{lem:positivity}, \eqref{eq:error25} and  Lemma \ref{lem:err-energy}. By using the estimate
\[
|\pa_t X|  \le C \left|\left( \dot\alpha,   \al \lda \dot a , \frac{ \al  \dot \lda }{\lda } \right)\right| (u+\lda^{\frac{2-n}{2}})
\]
in  Lemma \ref{lem:kernelpt} and \eqref{eq:error25}, we have
\[
| \langle w, \pa_t X_{j}\rangle_{\Lt} | \le  C(\|w\| + \lda^{1-n} )\left|\left( \dot\alpha,   \al \lda \dot a , \frac{ \al  \dot \lda }{\lda } \right)\right|\le C(M_2^{1/2}+\lda^{\frac{1-n}{2}}) \left|\left( \dot\alpha,   \al \lda \dot a , \frac{ \al  \dot \lda }{\lda } \right)\right|.
\]
By using \eqref{eq:almost-orth-basis}, that is, $X_0,\dots, X_{n+1}$ is almost an orthogonal basis of 
$E^{ut} \oplus E^c$, the lemma then follows by solving a system of linear equations.
\end{proof}

Next, we shall show that the projection onto the quasi-central space is the leading term of $M_{2}$. We start from a lemma.

\begin{lem} \label{lem:central-derivative}
For every $j=1,\dots, n+1$, there holds
\[
\frac{\ud b_j }{\ud t} =o(1) M_2^{1/2}+O(\lda^{-n}),
\]
where $o(1)\to 0$ as $t\to \infty$.
\end{lem}

\begin{proof}  For $j=1,\dots ,n+1$, by Lemma \ref{lem:ppty},
 \begin{align*}
-\frac{\ud b_j}{\ud t} &=\frac{\ud }{\ud t} \int_{\om}  (\mathcal{R}-r) u^{\frac{n+2}{n-2}}  X_j\,\ud x \\&
=\int_{\om} \left[\frac{n-2}{n+2} \Delta_g (\mathcal{R}-r) +\frac{4}{n+2} (\mathcal{R}-r)^2 +\frac{4}{n+2} r(\mathcal{R}-r) -\dot{r} \right] u^{\frac{n+2}{n-2}}  X_j\,\ud x
\\& \quad - \int_{\om}   (\mathcal{R}-r)^2 u^{\frac{n+2}{n-2}}  X_j \,\ud x+\int_{\om}  (\mathcal{R}-r) u^{\frac{n+2}{n-2}}  \pa_t  X_j\,\ud x
\end{align*}
and
\[
\dot{r} \int_{\om }u^{\frac{n+2}{n-2}}  X_j\,\ud x= O(M_2).
\]
Using the integrating by parts free formula \eqref{eq:integrationbyparts} and 
the conformal transformation law \eqref{eq:conf-law}, we have
 \begingroup
\allowdisplaybreaks
\begin{align*}
&\int_\om  u^{\frac{n+2}{n-2}}  X_j \Delta_g (\mathcal{R}-r) \,\ud x\\&= \int_\om  \left[(\Delta_g-\mathcal{R} ) (\mathcal{R}-r)\right] \frac{ X_j}{u} \,\ud vol_g +\int_\om \mathcal{R} (\mathcal{R}-r) u^{\frac{n+2}{n-2}} X_j \,\ud x  \\&
=  \int_\om   (\mathcal{R}-r) (\Delta_g -\mathcal{R}) \left(\frac{ X_j}{u}\right)  \,\ud vol_g  +\int_\om   (\mathcal{R}-r)^2 u^{\frac{n+2}{n-2}} X_j \,\ud x+\int_\om r (\mathcal{R}-r) u^{\frac{n+2}{n-2}} X_j \,\ud x  \\&
=\int_\om  (\mathcal{R}-r) u  \Delta X_j\,\ud x  +\int_\om   (\mathcal{R}-r)^2 u^{\frac{n+2}{n-2}} X_j \,\ud x+\int_\om r (\mathcal{R}-r) u^{\frac{n+2}{n-2}} X_j \,\ud x.
\end{align*}
\endgroup
By Lemma \ref{lem:kernelpt} and Lemma \ref{lem:tangent-flow},
\begin{align*}
\int  |(\mathcal{R}-r) u^{\frac{n+2}{n-2}}  \pa_t  X_j| \,\ud x& \le C \int  |(\mathcal{R}-r)| u^{\frac{n+2}{n-2}}(u+\lda^{\frac{2-n}{2}})\ud x \cdot |( \dot\alpha,   \al \lda \dot a , \frac{ \al  \dot \lda }{\lda } )|\\&
\le CM_2 ^{1/2}(M_2^{1/2}+\lda^{\frac{1-n}{2}}).
\end{align*}
Therefore,
 \begin{align*}
\frac{\ud b_j}{\ud t} &=-\frac{n-2}{n+2}\int_\om  (\mathcal{R}-r) u  \Delta X_j\,\ud x -\int_\om r (\mathcal{R}-r) u^{\frac{n+2}{n-2}} X_j \,\ud x+O(M_2).
\end{align*}

By \eqref{eq:almost-central},
\begin{align*}
-\frac{n-2}{n+2}\int_\om  (\mathcal{R}-r) u  \Delta X_j\,\ud x&=r_\infty  \int_\om (\mathcal{R}-r) u \overline X_0^{\frac{4}{n-2}} (X_j +O(\lda^{\frac{2-n}2})) \,\ud x.
\end{align*}
The right-hand side can be estimated as follows. We have
\begin{align*}
 \int_\om (\mathcal{R}-r) u \overline X_0^{\frac{4}{n-2}} X_j \,\ud x= \int_\om (\mathcal{R}-r) u (\overline X_0^{\frac{4}{n-2}}- u^{\frac{4}{n-2}})X_j \,\ud x+ \int_\om (\mathcal{R}-r)  u^{\frac{n+2}{n-2}}X_j \,\ud x.
\end{align*}
Since
 \begingroup
\allowdisplaybreaks
\begin{align*}
&\left| \int_\om (\mathcal{R}-r) u (\overline X_0^{\frac{4}{n-2}}- u^{\frac{4}{n-2}})X_j \,\ud x\right|\\
&\le \left| \int_\om (\mathcal{R}-r) u (\overline X_0^{\frac{4}{n-2}}- X_0^{\frac{4}{n-2}})X_j \,\ud x\right|+\left| \int_\om (\mathcal{R}-r) u (X_0^{\frac{4}{n-2}}- u^{\frac{4}{n-2}})X_j \,\ud x\right|\\
&\le C \int_\om |\mathcal{R}-r|(h_{a,\lda}X_0^\frac{4}{n-2}+ u h_{a,\lda}^\frac{4}{n-2}) |X_j|\,\ud x + o(1)\left| \int_\om (\mathcal{R}-r) u^{\frac{n+2}{n-2}}X_j \,\ud x\right|\\
&\le C \int_\om |\mathcal{R}-r|(\lda^{\frac{2-n}{2}} u^\frac{4}{n-2}+ \lda^{-2}u) (u+\lda^{\frac{2-n}{2}})\,\ud x +o(M_2^{1/2})\\
&\le C \lda^{-2}\int_\om |\mathcal{R}-r|u^2\,\ud x +C \lda^{2-n}\int_\om |\mathcal{R}-r|u^\frac{4}{n-2}\,\ud x +C \lda^{-\frac{n+2}{2}}\int_\om |\mathcal{R}-r|u\,\ud x +o(M_2^{1/2})\\
&\le C \lda^{-2} M_{\frac{n}{n-2}}^{\frac{n-2}{n}}+C \lda^{2-n} M_{\frac{n}{2}}^{\frac{2}{n}}+ C \lda^{-\frac{n+2}{2}}M_{\frac{2n}{n-2}}^\frac{n-2}{2n}+o(M_2^{1/2})\\
&\le C \lda^{-2} M_{2}^{\frac{n-2}{n}}+C \lda^{2-n} M_{2}^{\frac{2}{n}}+ C \lda^{-\frac{n+2}{2}}M_{2}^\frac{n-2}{2n}+o(M_2^{1/2})\\
&\le C \lda^{-n}+CM_2+o(M_2^{1/2}),
\end{align*}
\endgroup
where we used Proposition \ref{lem:relative-1} in the second inequality and Proposition \ref{prop:bubble-derivative} in the third inequality, and
\begin{align*}
\lda^{\frac{2-n}2}\int_\om |\mathcal{R}-r| u \overline X_0^{\frac{4}{n-2}}  \,\ud x&\le 
C\lda^{\frac{2-n}2}\int_\om |\mathcal{R}-r| u (u^{\frac{4}{n-2}}+\lda^{-2})  \,\ud x\\
&=o(M_2^{1/2}) +\lda^{-\frac{n+2}{2}} \int_\om |\mathcal{R}-r| u \,\ud x\\
&\le o(M_2^{1/2}) +C\lda^{-n}+CM_2,
\end{align*}
we obtain
\begin{align*}
\frac{\ud b_j}{\ud t} & = r_\infty  \int_\om  (\mathcal{R}-r) u^{\frac{n+2}{n-2}}  X_j\,\ud x - r \int_{\om} (\mathcal{R}-r) u^{\frac{n+2}{n-2}}  X_j \,\ud x + o(M_{2}^{1/2}) +O(\lda^{-n}) \\&
 =(r-r_\infty) \int_\om  (\mathcal{R}-r) u^{\frac{n+2}{n-2}}  X_j\,\ud x  + o(M_{2}^{1/2})+O(\lda^{-n}) \\&
 =o(M_{2}^{1/2}) +O(\lda^{-n}).
\end{align*}
The lemma is proved.
\end{proof}

\begin{prop}\label{prop:central-dominate} With $o(1)\to 0$ as $t\to \infty$,
\[
M_2-\sum_{j=1}^{n+1}\frac{b_j^2}{\kappa_j} = o(1)\left(\sum_{j=1}^{n+1} \frac{b_j^2}{\kappa_j} +\lda^{2-2n}\right),
\]
where $ \kappa_1,\dots \kappa_{n+1}$ are the positive constants in \eqref{eq:almost-orth-basis}.
\end{prop}

\begin{proof} The proof is inspired by that of Lemma 4.2 in Struwe \cite{St05}. 
Since $X_0,\dots, X_{n+1}$ is a basis of $E^{ut} \oplus E^c$, 
let us write
\begin{align}\label{eq:L2decom}
-(\mathcal{R}-r) u=\sum_{j=0}^{n+1} \beta_j(t) X_j +\eta_u \quad \mbox{with }\eta_u\in E^s. 
\end{align}
 By using \eqref{eq:almost-orth-basis}, $|b_j|\le CM_2^{1/2}$ and Lemma \ref{lem:err-energy}, we obtain $|\beta_j|\le CM_2^{1/2}$, 
\[
\beta_j \|X_j\|_{\Lt}^2 = b_j+O\Big((\|w\|+\lda^{2-n})M_{2}^{1/2}\Big)=b_j+O(M_{2}+\lda^{\frac{1-n}{2}}M_{2}^{1/2})
\]
and
\be \label{eq:m2-quant}
\begin{split}
M_2&= (1+o(1))\left(\sum_{j=0}^{n+1} \beta_j^2\|X_j\|_{\Lt}^2 + \|\eta_u\|_{\Lt}^2\right)= (1+o(1))\left(\sum_{j=1}^{n+1} \frac{b_j^2}{\kappa_j} + \|\eta_u\|_{\Lt}^2\right),
\end{split}
\ee
where we used $b_0^2=o(1)M_2$ by Lemma \ref{lem:unstable-small}.

For $j=1,\dots,n+1$, we have
\[
\int_\Omega \beta_j\nabla X_j \nabla \eta_u= -\int_\Omega \beta_j\Delta X_j \eta_u=r_\infty  \frac{n+2}{n-2}  \int_\Omega \beta_j\overline X_0 ^{\frac{4}{n-2}}  (X_j +O(\lda^{\frac{2-n}2 })) \eta_u .
\]
Since $\eta_u\in E^s$, 
 \begingroup
\allowdisplaybreaks
\begin{align*}
\int_\Omega  \overline X_0^{\frac{4}{n-2}} \beta_j  X_j \eta_u
&= \int_\Omega  (\overline X_0^{\frac{4}{n-2}}-u^{\frac{4}{n-2}}) \beta_j X_j \eta_u \\
&= \int_\Omega  (\overline X_0^{\frac{4}{n-2}}-X_0^{\frac{4}{n-2}})  \beta_j X_j \eta_u +\int_\Omega  \left[\left(\frac{X_0}{u}\right)^{\frac{4}{n-2}}-1\right]  u^{\frac{4}{n-2}}\beta_j X_j \eta_u\\
&= o(1)(\|\beta_j X_j\|_{L^2}^2 +\|\eta_u\|_{L^2}^2)+ o(1)(\|\beta_j X_j\|^2_{\Lt} + \|\eta_u\|^2_{\Lt})\\
&= o(1)(\|\beta_j X_j\|_{H^1_0}^2 +\|\eta_u\|_{H^1_0}^2)+ o(1)(\|\beta_j X_j\|^2_{\Lt} + \|\eta_u\|^2_{\Lt}),
\end{align*}
\endgroup
where we used Proposition \ref{lem:relative-1} in the third equality. Also,
\begin{align*}
\lda^{\frac{2-n}2 }\int_\Omega |\beta_j\overline X_0 ^{\frac{4}{n-2}} \eta_u| \le \lda^{\frac{2-n}2 }\int_\Omega \beta_j^2\overline X_0 ^{\frac{4}{n-2}} + \lda^{\frac{2-n}2 }\int_\Omega \eta_u^2 \overline X_0 ^{\frac{4}{n-2}} =o(1)\beta_j^2+\lda^{\frac{2-n}2 }\int_\Omega \eta_u^2 \overline X_0 ^{\frac{4}{n-2}}
\end{align*}
and
\begin{align*}
&\lda^{\frac{2-n}2 }\int_\Omega \eta_u^2 \overline X_0^{\frac{4}{n-2}}\\
&=\lda^{\frac{2-n}2 }\int_\Omega \eta_u^2 (\overline X_0^{\frac{4}{n-2}}-X_0^{\frac{4}{n-2}})+\lda^{\frac{2-n}2 }\int_\Omega \eta_u^2 (X_0^{\frac{4}{n-2}}-u^{\frac{4}{n-2}})+ \lda^{\frac{2-n}2 }\int_\Omega \eta_u^2 u^{\frac{4}{n-2}}\\
&=o(1) \|\eta_u\|_{H^1_0}^2+o(1)\|\eta_u\|^2_{\Lt}.
\end{align*}
Together with \eqref{eq:almost-orth-basis}  and \eqref{eq:auxH1norm}, we obtained
\[
\int_\Omega \beta_j\nabla X_j \nabla \eta_u=o(1)(\|\beta_j X_j\|^2_{\Lt} +\|\eta_u\|_{H^1_0}^2+ \|\eta_u\|^2_{\Lt}).
\]
Similarly, one can show
\[
\int_\Omega \beta_0\nabla X_0 \nabla \eta_u=o(1)(\|\beta_0 X_0\|^2_{\Lt} +\|\eta_u\|_{H^1_0}^2+ \|\eta_u\|^2_{\Lt}).
\]

Therefore, by taking the $L^2$ norm of the gradient of \eqref{eq:L2decom} and using Lemma \ref{lem:gradientnormest}, we have
\begin{align}
&\int_{\om} |\nabla [(\mathcal{R}-r) u]|^2\,\ud x\nonumber\\
&\ge (1+o(1)) \frac{n+2}{n-2} r_\infty \sum_{j=1}^{n+1}  \beta_j^2 \|X_j\|^2_{\Lt} +  (1+o(1)) \int_\Omega |\nabla \eta_u|^2 \,\ud x +o(1)\|\eta_u\|^2_{\Lt}\nonumber \\&
\ge  (1+o(1)) \frac{n+2}{n-2} r_\infty \sum_{j=1}^{n+1}  \beta_j^2 \|X_j\|^2_{\Lt} +\frac{1}{1-c/2}  \frac{n+2}{n-2}  r_\infty \|\eta_u\|^2_{\Lt}  , \label{eq:gradient-m2}
\end{align}
where we used Proposition \ref{lem:positivity} in the last inequality that gives the constant $c>0$.

 By Lemma \ref{lem:ppty} and the conformal transform law \eqref{eq:conf-law}, we have
\begin{align*}
\frac12 \frac{\ud }{\ud t} M_2(t)&=  \int_{\om } (\mathcal{R}-r) \pa_t (\mathcal{R}-r) \,\ud vol_g - \frac{n}{n+2} \int_{\om } (\mathcal{R}-r)^3 \,\ud vol_g  \\&
=- \frac{n-2}{n+2} \int_\Omega |\nabla [u(\mathcal{R}-r)]|^2 \,\ud x + \frac{2 }{n+2} \int_{\om } (\mathcal{R}-r)^3 \,\ud vol_g  + r M_2 (t)\\&
=- \frac{n-2}{n+2} \int_\Omega |\nabla [u(\mathcal{R}-r)]|^2 \,\ud x  + r_\infty M_2 (t)  +o(1)M_2,
\end{align*}
where we used $r-r_\infty=o(1)$ and  $ \int_{\om } (\mathcal{R}-r)^3 \,\ud vol_g= o(1) M_2$ by item (i) of Proposition \ref{prop:jx20}.  
Using \eqref{eq:m2-quant} and \eqref{eq:gradient-m2}, we immediately obtain
\be  \label{eq:Lyap-ineq}
\frac12 \frac{\ud }{\ud t} M_2(t) \le - C_0 \|\eta_u\|^2_{\Lt} +o(1)M_2, 
\ee
where $C_0= \frac{c r_\infty}{2(2-c)}>0$.

Let
\[
\widetilde B= \left( \frac{b_1}{\sqrt{\kappa_1}},\dots, \frac{b_{n+1}}{\sqrt{\kappa_{n+1}}}\right). 
\]
We claim that  if $|\widetilde B(t_1)|^2 +\lda^{2-2n} \ge  \|\eta_u(t_1)\|^2_{\Lt} $ for some large $t_1$, then
\[
(1+o(1))(|\widetilde B|^2+\lda^{2-2n}) = M_2 +\lda^{2-2n}\quad \mbox{as }t\to\infty.
\]
Indeed, by  \eqref{eq:m2-quant}, we have
\begin{align}\label{eq:m2-quantlda}
M_2+\lda^{2-2n}= (1+o(1))\Big(|\widetilde B|^2+\lda^{2-2n}+ \|\eta_u\|_{\Lt}^2\Big).
\end{align}
Then for $t$ near $t_1$, we may write
\[
M_2+\lda^{2-2n}= (1+\gamma(t)) (|\widetilde B|^2+\lda^{2-2n})\quad \mbox{with }-1/2<\gamma(t) < 2.
\]
By Lemma \ref{lem:tangent-flow}, we know $\dot{\lda}=o(1)\lda$. The using Lemma \ref{lem:central-derivative}, we have
\begin{align*}
 -C_0 \|\eta_u\|_{\Lt}^2 +o(1)(M_2+\lda^{2-2n}) &\ge \frac{\ud }{\ud t} (M_2+\lda^{2-2n})\\&= (|\widetilde B|^2+\lda^{2-2n}) \frac{\ud \gamma}{\ud t} +2(1+\gamma) \widetilde B\cdot\frac{\ud }{\ud t} \widetilde B+o(\lda^{2-2n})\\
 & = (|\widetilde B|^2 +\lda^{2-2n})   \frac{\ud \gamma}{\ud t} +O(M_2^{1/2}\lda^{-n})+o(1)(M_2+\lda^{2-2n})\\
 & = (|\widetilde B|^2 +\lda^{2-2n})   \frac{\ud \gamma}{\ud t} +o(1)(M_2+\lda^{2-2n}).
\end{align*}
Making use of \eqref{eq:m2-quantlda}, we have
\[
(|\widetilde B|^2+\lda^{2-2n}) \frac{\ud \gamma}{\ud t} \le -(C_0 \gamma(t)+o(1))(|\widetilde B|^2+\lda^{2-2n}).
\]
Canceling the factor $|\widetilde B|^2+\lda^{2-2n}$, we find
\[
 \frac{\ud \gamma}{\ud t} \le -(C_0 \gamma(t)+o(1)).
\]
Then $\gamma(t)\to 0$ as $t\to\infty$, and the claim follows.

It remains to show that $|\widetilde B(t_1)|^2 +\lda(t_1)^{2-2n} \ge \|\eta_u(\cdot,t_1)\|^2_{\mathcal{L}^{2}_{t_1}}$ for some arbitrarily large $t_1$. If not, then \eqref{eq:Lyap-ineq} and \eqref{eq:m2-quantlda} yield that
\[
\frac{\ud }{\ud t} (M_2 +\lda^{2-2n}) \le -(C_0+o(1)) (M_2+\lda^{2-2n}).
\]
It follows that $M_2(t) +\lda(t)^{2-2n}\le Ce^{-\frac{C_0}{2} t}$ for large $t$, which particularly implies
\[
M_2^{1/2}(t) \in L^1[1, \infty).
\]
By \eqref{eq:stable-1}, $u(\cdot, t)$ is a Cauchy sequence in $H_0^1$ and thus converges to a steady solution.  We obtained a contradiction. The proposition is proved. 
\end{proof}

\subsection{Estimate of the projection to the quasi-central space}

In this subsection, we shall estimate
\[
b_1,\dots, b_{n+1}
\]
via the $(n+1)$ identities in the proof of the Pohozaev identity. 

\begin{prop}\label{prop:Quasi-central} We have
\begin{align*}
M_2&= O(\lda^{2(2-n)}),\\
b_j&= o(\lda^{2-n}) , \quad j=1,\dots,n,  \\
b_{n+1}&=\overline C(n) H(a,a)\lda^{2-n}+o(\lda^{2-n}),
\end{align*}
where $\overline C(n)>0$ is a constant depending only on $n$, and $H>0$ is as in Proposition \ref{prop:bubble-derivative}.
\end{prop}

\begin{proof}
By the definition of $\mathcal{R}$, we have 
\be \label{eq:fake-elliptic}
-\Delta u=\mathcal{R} u^{\frac{n+2}{n-2}} \quad \mbox{in }\om, \quad u=0 \quad \mbox{on }\pa \om.
\ee
By  Proposition \ref{lem:relative-1}, and local estimates of the Poisson equation, we have 
\[
|u|(x)+|\nabla u(x)| \le C \lda^{\frac{2-n}{2}} \quad \forall ~d(x)<\frac{\delta_0}{10},
\]
where $\delta_0$ is the constant  in Proposition \ref{prop:jx20}. 

First, multiplying $\frac{\pa u }{\pa x_j} $, $j=1,\dots, n$,  to \eqref{eq:fake-elliptic} and integrating by parts, we have 
 \begingroup
\allowdisplaybreaks
\begin{align*}
\frac{n-2}{2n}\int_{\om } \mathcal{R} \frac{\pa }{\pa x_j} u^{\frac{2n}{n-2}} \,\ud x&=  \int_\Omega \mathcal{R} u^{\frac{n+2}{n-2}} \frac{\pa u}{\pa x_j}\,\ud x \\& =  -\int_{\om} \Delta u \cdot \frac{\pa u}{\pa x_j}\,\ud x\\&
=-\frac{1}{2}  \int_{\pa \om} (\pa_\nu u)^2 \cdot \nu_j\,\ud S\\&
=O(\lda^{2-n}).
\end{align*}
\endgroup
Hence,
 \begingroup
\allowdisplaybreaks
\begin{align*}
O(\lda^{2-n})&=\frac{n-2}{n}\int_{\om }    u^{\frac{2n}{n-2}} \nabla_x  \mathcal{R} \,\ud x \\
&=\frac{n-2}{n}\int_{\om }    u^{\frac{2n}{n-2}} \nabla_x  (\mathcal{R}-r) \,\ud x \\
&=-2\int_{\om }  ( \mathcal{R}-r)  u^{\frac{n+2}{n-2}} (\alpha \nabla_x  X_0 +\nabla_x w)\,\ud x\\&
=-2 \alpha \lda (b_1,\dots, b_n) -2\int_{\om } ( \mathcal{R}-r)  u^{\frac{n+2}{n-2}} \nabla w\,\ud x+O(M_2+\lda^{1-n}),
\end{align*}
\endgroup
where we used $\nabla_x  PU_{a,\lda}=-\nabla_a PU_{a,\lda}+O(\lda^{-\frac{n-2}{2}})$ thanks to Proposition  \ref{prop:bubble-derivative}, and 
\begin{align}\label{eq:erroelda}
\lda^{\frac{2-n}{2}}\int_{\om }  |\mathcal{R}-r| u^{\frac{n+2}{n-2}} \,\ud x\le \int_{\om }  (\mathcal{R}-r)^2 u^{\frac{2n}{n-2}} \,\ud x + \lda^{2-n}  \int_{\om } u^{\frac{4}{n-2}} \,\ud x = M_2 + O(\lda^{1-n}).
\end{align}
Using Lemma \ref{lem:H2}, we can find
\[
\left|\int_{\om } ( \mathcal{R}-r)  u^{\frac{n+2}{n-2}} \nabla w\,\ud x\right|\le C M_2^{1/2} \|u\|^{1/n}_{L^{\frac{2n}{n-2}}} \|\nabla w\|_{L^\frac{2n}{n-2}}=o(1) M_{2}^{1/2} \lda.
\]
Therefore,
\begin{equation}\label{eq:bj-1}
b_j= o(1) M_2^{1/2} +O(\lda^{1-n}) , \quad j=1,\dots,n.
\end{equation}

Second, taking $(x-a)\cdot \nabla_x u$ as a test function  against \eqref{eq:fake-elliptic} we obtain
 \begingroup
\allowdisplaybreaks
\begin{align*}
&\int_{\om }  (\mathcal{R}-r) u^{\frac{n+2}{n-2}} (x-a)\cdot \nabla u\,\ud x\\&= -\int_{\om }   (x-a)\cdot \nabla u \Delta u\,\ud x- r\frac{n-2}{2n}\int_{\om }   (x-a)\cdot \nabla u^{\frac{2n}{n-2}}\,\ud x\\&
=-\frac{n-2}{2} \int_{\om }  |\nabla u|^2 \,\ud x-\frac{1}{2} \int_{\pa \om} |\nabla u|^2 \langle (x-a), \nu\rangle \,\ud S+r\frac{n-2}{2}\int_{\om }  u^{\frac{2n}{n-2}}\,\ud x\\&
= -\frac{1}{2} \int_{\pa \om} |\nabla u|^2 \langle (x-a), \nu\rangle \,\ud S,
\end{align*}
\endgroup
where we used the definition of $r$ in the last equality. 
Note that
 \begingroup
\allowdisplaybreaks
\begin{align*}
 (x-a) \cdot \nabla  PU_{a,\lda}&= (x-a) \cdot \nabla ( U_{a,\lda}-h_{a,\lda}) \\&
 = (2-n)   U_{a,\lda} \frac{\lda^2 |x-a|^2 }{1+\lda^2 |x-a|^2} - (x-a) \cdot \nabla h_{a,\lda}\\&
 = \frac{n-2}{2}  U_{a,\lda} \frac{1-\lda^2 |x-a|^2 }{1+\lda^2 |x-a|^2}+\frac{2-n}{2}  U_{a,\lda}  - (x-a) \cdot\nabla h_{a,\lda}\\&
 =\lda \pa_\lda  PU_{a,\lda}+\lda \pa_\lda h_{a,\lda} +\frac{2-n}{2}  U_{a,\lda} - (x-a) \cdot \nabla h_{a,\lda}\\&
 =r_\infty^{\frac{n-2}{4}} ( X_{n+1} +\frac{2-n}{2} X_0) +(\lda \pa_\lda h_{a,\lda}+\frac{2-n}{2} h_{a,\lda}- (x-a) \cdot \nabla h_{a,\lda} )\\
 & =r_\infty^{\frac{n-2}{4}} ( X_{n+1} +\frac{2-n}{2} X_0) + O(\lda^{\frac{2-n}{2}}),
\end{align*}
\endgroup
where we used Proposition \ref{prop:bubble-derivative} in the last inequality. It follows that
\[
(x-a)\cdot \nabla u=  \alpha (X_{n+1}-\frac{n-2}{2} X_0) +   (x-a)\cdot  \nabla w +O(\lda^{\frac{2-n}{2}}).
\]
Hence,
\begin{align*}
&\int_{\om }  (\mathcal{R}-r) u^{\frac{n+2}{n-2}} (x-a)\cdot \nabla u\,\ud x \\
&=-\alpha  b_{n+1} +\frac{(n-2)\alpha}{2} b_0+\int_{\om }  (\mathcal{R}-r) u^{\frac{n+2}{n-2}} (x-a)\cdot \nabla w\,\ud x+O(\lda^{\frac{2-n}{2}})\int_{\om }  (\mathcal{R}-r) u^{\frac{n+2}{n-2}} \,\ud x.
\end{align*}
Since  $u^{\frac{4}{n-2}}|x-a|^2\le C \overline X_{0}^{\frac{4}{n-2}}|x-a|^2\le C$, we have 
\begin{align*}
&\left|\int_{\om }  (\mathcal{R}-r) u^{\frac{n+2}{n-2}} (x-a)\cdot \nabla w\,\ud x\right| \le  \int_{\om }  (\mathcal{R}-r)^2 u^{\frac{2(n+2)}{n-2}} (x-a)^2 \,\ud x + \|w\|^2\le CM_2 + \|w\|^2.
\end{align*}
Thus,  by Lemma \ref{lem:err-energy}, Lemma \ref{lem:unstable-small} and \eqref{eq:erroelda}, we have
\[
\int_{\om }  (\mathcal{R}-r) u^{\frac{n+2}{n-2}} (x-a)\cdot \nabla u\,\ud x =-\alpha  b_{n+1} +O(M_2+\lda^{1-n}).
\]
 Therefore,
\begin{align*}
b_{n+1}&=\frac{1}{2\alpha } \int_{\pa \om} |\nabla u|^2 \langle (x-a), \nu\rangle \,\ud S+O(M_2+\lda^{1-n}).
\end{align*}
By Lemma \ref{lem:traceH1}, we have 
\[
\int_{\pa \om} |\nabla w|^2 \langle (x-a), \nu\rangle\,\ud S=O(M_2+\lda^{1-n}).
\]
Thus, we have
\begin{align}
&\int_{\pa \om} |\nabla u|^2 \langle (x-a), \nu\rangle\,\ud S\nonumber \\
&= \alpha^2 \int_{\pa \om} |\nabla  X_0|^2 \langle (x-a), \nu\rangle \,\ud S+ 2\alpha \int_{\pa \om} \nabla  X_0\cdot \nabla w \langle (x-a), \nu\rangle \,\ud S+O(M_2+\lda^{1-n})\nonumber\\
&= C_2(n)\alpha^2 r_\infty^{-\frac{n-2}{2}} H(a,a)\lda^{2-n}+O\Big(\lda^\frac{2-n}{2}(M_2^{1/2}+\lda^{\frac{1-n}{2}})\Big)+O(M_2+\lda^{1-n}),\label{eq:bn+1}
\end{align}
where we used \eqref{eq:phoest} in the last inequality, and 
\[
C_2(n)= \frac{(n-2)(n+2)}{n}[ n(n-2)]^{\frac{n+ 2}{4}} \int_{\R^n} (1+|x|^2)^{-\frac{n+4}{2}} |x|^2\,\ud x.  
\]
This in particular implies that
\[
b_{n+1}= O(M_2+\lda^{2-n}).
\]
Together with \eqref{eq:bj-1} and Proposition \ref{prop:central-dominate},  we have
\[
M_2= (1+o(1)) \sum_{j=1}^{n+1} \frac{b_j^2}{\kappa_j} +o(\lda^{2-2n})= o(1)M_2+ O(M_2^2 +\lda^{2(2-n)}).
\]
Since $M_2\to 0$, we obtain
\[
M_2= O(\lda^{2(2-n)}).
\]
Then the conclusion of this lemma follows from plugging this estimate of $M_2$ to \eqref{eq:bj-1} and \eqref{eq:bn+1}, with the help of the estimate $|\alpha-1|$ in Lemma \ref{lem:alpha-r}.
\end{proof}

\begin{prop}\label{prop:final} 
There exist $a_\infty\in \om$ with $d(a_\infty)>\delta_0/2$,  and $C_3(n)>0$ depending only on $n$, such that 
\begin{equation}\label{eq:aldarate}
\begin{split}
 |a(t)-a_\infty |&=o(t^{-\frac{1}{n-2}}), \quad\mbox{and} \\
\lim_{t\to\infty}t^{-\frac{1}{n-2}} \lda(t)&= C_3(n)H(a_\infty,a_\infty). 
\end{split}
\end{equation}

\end{prop}

\begin{proof} 
By Lemma \ref{lem:tangent-flow} and  Proposition \ref{prop:Quasi-central}, we have
\be \label{eq:lda-ode-1}
\begin{split}
\frac{\dot \lda}{\lda}&= \left(\frac{n+2}{n-2} +o(1)\right) \frac{b_{n+1}}{\kappa_{n+1}} + o(\lda^{2-n})= ( C_3(n) H(a,a)+o(1))\lda^{2-n},
\end{split}
\ee
where $C_3(n)=\frac{(n+2)\overline C(n)}{(n-2)\kappa_{n+1}}$. It follows that 
\[
\frac{1}{C} t^{\frac{1}{n-2}} \le  \lda\le C t^{\frac{1}{n-2}}.
\]
Using Lemma \ref{lem:tangent-flow} again,
\[
\dot a=o(1)\lda^{1-n}= o(1) t^{-\frac{n-1}{n-2}} \in L^1([1,\infty).
\]
Hence, there exists $a_\infty\in \om$ such that
\[
|a(t)-a_\infty| =o(1) t^{-\frac{1}{n-2}} =o(1) \lda^{-1}.
\]
Applying  this fact into \eqref{eq:lda-ode-1}, we obtain 
\[
\frac{\dot \lda}{\lda}= (C_3(n) H(a_\infty,a_\infty)+o(1))\lda^{2-n}. 
\]
The proposition follows.
\end{proof}

Finally, we prove the main theorem of this paper.
\begin{proof}[Proof of Theorem \ref{thm:main}]
Recall the change of variables in \eqref{eq:from innormal to normal}. It follows from \eqref{eq:BN-5} that
\[
\frac{1}{C}\le \frac{\beta(s)}{s}\le C.
\]
From \eqref{eq:definitionrd} and \eqref{eq:definitionr}, we have
\[
\frac{\ud }{\ud t} f(t)=-\frac{2n}{n+2} r(\beta^{-1}(t)) f(t)^\frac{n-2}{n}+\frac{2n}{n+2}f(t),
\]
where
\[
f(t)=\int_\Omega v(x,t)^{\frac{2n}{n-2}}\,\ud x,
\]
and $\beta^{-1}$ is the inverse function of $\beta$. Hence,
\[
\frac{\ud }{\ud t}(e^{-\frac{4}{n+2}t}f(t)^{\frac{2}{n}})=-\frac{4}{n+2}r(\beta^{-1}(t)) e^{-\frac{4}{n+2}t}.
\]
This implies that
\begin{align*}
e^{-\frac{4}{n+2}t}f(t)^{\frac{2}{n}}&=\int_t^\infty\frac{4}{n+2}(r(\beta^{-1}(\sigma))-r_\infty) e^{-\frac{4}{n+2}\sigma}\,\ud \sigma+\int_t^\infty\frac{4}{n+2}r_\infty e^{-\frac{4}{n+2}\sigma}\,\ud \sigma\\
&=\int_t^\infty\frac{4}{n+2}(r(\beta^{-1}(\sigma))-r_\infty) e^{-\frac{4}{n+2}\sigma}\,\ud \sigma+e^{-\frac{4}{n+2}t}r_\infty.
\end{align*}
By Lemma \ref{lem:alpha-r}, Proposition \ref{prop:Quasi-central} and and Proposition \ref{prop:final}, we obtain
\begin{align*}
e^{-\frac{4}{n+2}t}f(t)^{\frac{2}{n}}-e^{-\frac{4}{n+2}t}r_\infty&=\int_t^\infty\frac{4}{n+2}(r(\beta^{-1}(\sigma))-r_\infty) e^{-\frac{4}{n+2}\sigma}\,\ud \sigma\\
&\le C(\beta^{-1}(t))^{-1/2} \int_t^\infty\frac{4}{n+2}e^{-\frac{4}{n+2}\sigma}\,\ud \sigma\\
&\le Ct^{-1/2} e^{-\frac{4}{n+2}t}.
\end{align*}
Hence, $0\le f(t)^{\frac{2}{n}}-r_\infty\le Ct^{-1/2}$. That is,
\begin{align}\label{eq:ratevnorm}
0\le \|v(\cdot,t)\|_{L^\frac{2n}{n-2}}-r_\infty^{\frac{n-2}{4}}\le Ct^{-1/2}.
\end{align}

Since $Y_\om (u(\cdot, 0))=Y_\om (\rho_0^m) \le 2^{\frac{2}{n}}K(n)$, then we have the dichotomy in Corollary \ref{cor:alternative}. 

If part (i) in Corollary \ref{cor:alternative} happens, then it follows from \eqref{eq:uinfinityconvergence},  \eqref{eq:from innormal to normal} and \eqref{eq:ratevnorm} that
\[
\left\|\frac{v(\cdot,t)}{S}-1\right\|_{C^2(\overline\Omega)}\le C_1 t^{-\gamma_1}
\]
for some $C_1, \gamma_1>0$, where $S=r_\infty^{\frac{n-2}{4}}u_\infty$ satisfying \eqref{eq:elliptic}. Then, the estimate \eqref{eq:sr-infinity} follows from the change of variable in \eqref{eq:changing_variables}.

If part (ii) in Corollary \ref{cor:alternative} happens, then it follows from Proposition \ref{lem:relative-1}, Proposition \ref{prop:Quasi-central} and Proposition \ref{prop:final}, as well as \eqref{eq:ratevnorm} and the change of variables in  \eqref{eq:from innormal to normal}, that
\[
\left\|\frac{v(\cdot,t)}{PU_{a(t), \lda(t)}}-1\right\|_{L^{\infty}(\om)}\le C_2 t^{-\gamma_2},
\]
for some $C_2, \gamma_2>0$, where $a: (0,T^*)\to \om$ and $\lda:(0,T^*) \to [1,\infty) $ are smooth functions satisfying \eqref{eq:aldarate}. Then the estimates \eqref{eq:sr-1} and \eqref{eq:bubblelimit} follow from the change of variables in \eqref{eq:changing_variables}. 
\end{proof}

\bigskip

\noindent T. Jin

\noindent Department of Mathematics, The Hong Kong University of Science and Technology\\
Clear Water Bay, Kowloon, Hong Kong.

\noindent  \textsf{Email: tianlingjin@ust.hk}

\medskip

\noindent J. Xiong

\noindent School of Mathematical Sciences, Beijing Normal University\\
Beijing 100875, China.

\noindent \textsf{Email: jx@bnu.edu.cn}

\end{document}